\documentclass[11pt]{amsart}
\usepackage[hmargin=3cm,vmargin=3cm]{geometry}

\usepackage[latin1]{inputenc}
\usepackage{xspace,amssymb,amsfonts,euscript, palatino}
\usepackage{amsthm,amsmath,amscd,stmaryrd,latexsym}

\usepackage{graphicx}
\usepackage[all]{xy}
\usepackage{tikz}
\usetikzlibrary{calc, matrix, arrows, decorations.pathmorphing}
\usepackage[dvips]{epsfig}
\usepackage{pinlabel}

\newcommand{\ig}[2]{\vcenter{\xy (0,0)*{\includegraphics[scale=#1]{fig/#2}} \endxy}}

\newcommand{\igc}[2]{\begin{center} \includegraphics[scale=#1]{fig/#2} \end{center}}

\RequirePackage{color}
\definecolor{myred}{rgb}{0.75,0,0}
\definecolor{mygreen}{rgb}{0,0.5,0}
\definecolor{myblue}{rgb}{0,0,0.65}

\RequirePackage{ifpdf}
\ifpdf
  \IfFileExists{pdfsync.sty}{\RequirePackage{pdfsync}}{}
  \RequirePackage[pdftex,
   colorlinks = true,
   urlcolor = myblue, 
   citecolor = mygreen, 
   linkcolor = myred, 
   pagebackref,
   bookmarksopen=true]{hyperref}
\else
  \RequirePackage[hypertex]{hyperref}
\fi

\newtheorem{thm}{Theorem}[section]
\newtheorem{lemma}[thm]{Lemma}

\newtheorem{prop}[thm]{Proposition}
\newtheorem{cor}[thm]{Corollary}

\newtheorem*{prop*}{Proposition}

\theoremstyle{definition}
\newtheorem{defn}[thm]{Definition}

\newtheorem{notation}[thm]{Notation}

\newtheorem{example}[thm]{Example}

\newtheorem*{defn*}{Definition}

\theoremstyle{remark}
\newtheorem{remark}[thm]{Remark}

\numberwithin{equation}{section}

    \def\NM{{\mathbb{N}}}

    \def\AC{{\mathcal{A}}}
    
    \def\CC{{\mathcal{C}}}

\def\HB{{\mathbf H}}

    \def\NC{{\mathcal{N}}}

    \def\RC{{\mathcal{R}}}
    \def\SC{{\mathcal{S}}}

\def\a{\alpha}
\def\b{\beta}

\def\G{\Gamma}

\def\l{\lambda}

\def\s{\sigma}

\let\phi=\varphi

\usepackage{bbm}

\def\Z{{\mathbbm Z}}

\def\1{\mathbbm{1}}

\newcommand{\un}{\underline}

\newcommand{\ot}{\otimes}
\newcommand{\pa}{\partial}
\newcommand{\co}{\colon}
\newcommand{\sqot}{\boxtimes}

\renewcommand{\to}{\rightarrow}
\newcommand{\into}{\hookrightarrow}

\renewcommand{\sl}{\mathfrak{sl}}

\newcommand{\refequal}[1]{\xy {\ar@{=}^{#1}
(-1,0)*{};(1,0)*{}};
\endxy}

\DeclareMathOperator{\Hom}{{\rm Hom}}

\DeclareMathOperator{\End}{{\rm End}}

\DeclareMathOperator{\id}{{\rm id}}

\DeclareMathOperator{\irr}{\rm{irr}}

\DeclareMathOperator{\Ob}{\rm{Ob}}

\newcommand{\SymCat}{\SC {ym}}

\begin{document}

\begin{abstract}

In this paper we give a version of Bergman's diamond lemma which applies to certain monoidal categories presented by generators and relations. In particular, it applies to: the Coxeter
presentation of the symmetric groups, the quiver Hecke algebras of Khovanov-Lauda-Rouquier, the Webster tensor product algebras, and various generalizations of these.

We also give an extension of Manin-Schechtmann theory to non-reduced expressions.
\end{abstract}

\title{A diamond lemma for Hecke-type algebras}

\author{Ben Elias} \address{University of Oregon, USA}

\maketitle

\section{Introduction}
\label{sec:intro}

When presenting an algebra by generators and relations, it is not obvious how large this algebra will be. Even with innocuous-looking relations, it may turn out that the algebra was zero
all along. It is often possible to find a set of \emph{monomials}, i.e. elements of the algebra expressed as words in the generators, which span the algebra, and which one expects is a
basis. However, proving the linear independence of this set of monomials can be quite a difficult task.

For example, consider the so-called \emph{Coxeter presentation} of the symmetric group $S_n$ or of its group algebra: it is generated by symbols $s_i$, $1 \le i \le n-1$, modulo the relations
\begin{subequations} \label{subeq:coxeter}
\begin{equation} \label{eq:s2} s_i s_i = 1, \end{equation}
\begin{equation} \label{eq:sisj} s_i s_j = s_j s_i \quad \textrm{ for } j \ne i \pm 1, \end{equation}
\begin{equation} \label{eq:sisjsi} s_i s_j s_i = s_j s_i s_j \quad \textrm{ for } j = i \pm 1. \end{equation}
\end{subequations}
Fix an arbitrary reduced expression $\un{w} = s_{i_1} s_{i_2} \cdots s_{i_d}$ for each $w \in S_n$. After a great deal of combinatorics, one can prove that these particular reduced expressions span the group algebra. Now one can ask whether these reduced expressions are linearly independent.

One way to check linear independence is to find an action of the algebra where the elements act by linearly independent operators. For example, we can consider the standard action of the
symmetric group $S_n$ on a set of size $n$, and linearize it to get a representation of the group algebra. Clearly the action of distinct permutations is linearly independent, from which
we deduce that our chosen reduced expressions form a basis. This is quite painless.

However, there are numerous contexts where it is not as easy to find an action of the algebra. Consider the generalized Hecke algebra $\HB(\a,\b)$ of $S_n$ (see
\cite[\S 7.1]{HumpCox}), which replaces \eqref{eq:s2} with \begin{equation} \label{eq:s2Hecke} s_i s_i = \a s_i + \b \end{equation} for certain parameters $a$ and $b$ in the base ring.
The same combinatorics as above (with an additional inductive argument by the length of an expression) implies that our chosen reduced expressions will span $\HB(\a,\b)$. It is still
the case that our chosen reduced expressions will form a basis, but finding an action which illustrates this can be hard.

For another example, consider the quiver Hecke algebra defined by Khovanov-Lauda \cite{KhoLau09} and Rouquier \cite{Rouq2KM-pp}. It has a presentation similar in style to \eqref{subeq:coxeter}
above, with extra generators $x_1, \ldots, x_n$ forming a polynomial algebra, and some additional relations involving them. A sample relation in the nilHecke algebra (a special case of
the quiver Hecke algebra) is \begin{equation} \label{eq:NH} x_i s_i - s_i x_{i+1} = s_i x_i - x_{i+1} s_i = 1. \end{equation} One expects there to be a basis of the form $\{f\un{w}\}$,
where $f$ lives in a basis of the polynomial algebra, and $w \in S_n$ (and $\un{w}$ is a fixed reduced expression for $w$). That this is a basis is proven by constructing the
``polynomial representation'' of the quiver Hecke algebra, and confirming that the suspected basis acts in a linearly independent fashion. However, in forthcoming work of the author
with Agustin Garcia \cite{EGarcia1}, we will construct variations on the quiver Hecke algebra which have no polynomial representations, and as such we need a different tool to prove
linear independence.

One rather painful approach is to construct an abstract version of the regular representation. If one has a set $X_{\irr}$ of monomials suspected to be a basis, one creates an abstract
vector space with basis given by $X_{\irr}$. One works out how each generator acts on this vector space, and then checks that the relations are satisfied. It then becomes obvious that
the suspected basis acts on this abstract vector space by linearly independent operators. While this approach works in theory for any presentation, computing how the generators act and
checking the relations can be a prohibitive amount of work. In special cases it can be done nonetheless. This is the approach taken by Humphreys for the Hecke algebra $\HB(\a,\b)$ in
\cite[\S 7.1 through 7.3]{HumpCox}, or for the PBW basis of the universal enveloping algebra in \cite[\S 17.4]{HumpLie}.

Thankfully, there is a tool designed for precisely this purpose, the Bergman diamond lemma. In some sense, one can view this diamond lemma as an efficient tool to check that the action
on the abstract regular representation is well-defined. It is a valuable tool for algebraists, and if you, dear reader, have not yet seen it, it is well worth your time to learn it! See \S\ref{sec:original} for a summary.

\begin{remark} The history of the Bergman diamond lemma vis a vis the related idea of Gr\"obner bases is somewhat complex, as they were pursued independently for a time. See \S\ref{subsec:vsGrobner} for a comparison, and for a discussion of previous work related to the results of this paper. \end{remark}

The Bergman diamond lemma fixes a direction for each relation, and only permits it to be applied in that direction. For example, in the Coxeter presentation for $S_3$, we might always send $s_1 s_1 \mapsto 1$, $s_2 s_2 \mapsto 1$, and $s_1 s_2 s_1 \mapsto s_2 s_1 s_2$. The monomials $\{1, s_1, s_2, s_1 s_2, s_2 s_1, s_2 s_1 s_2\}$ are \emph{irreducible} in that no relation can be applied to them, while every other monomial has either $s_1 s_1$, $s_2 s_2$, or $s_1 s_2 s_1$ as a subword. Then one asks which overlaps can occur between relations, and whether such an overlap can be resolved using the relations. For example, $s_1 {\bf s_1} s_2 s_1$ contains two overlapping relations, overlapping in the bold $s_1$, and hence is ambiguous because it can be hit by a relation in two different ways:
\[ s_1 s_1 s_2 s_1 \mapsto s_2 s_1\]
and
\[ s_1 s_1 s_2 s_1 \mapsto s_1 s_2 s_1 s_2 \mapsto s_2 s_1 s_2 s_2 \mapsto s_2 s_1.\]
Since these two ways converged to the same element $s_2 s_1$, we say that the ambiguity is resolvable. For any words $x$ and $y$, the word $x s_1 s_1 s_2 s_1 y$ also has an ambiguity, but this one is resolvable because the minimal ambiguity $s_1 s_1 s_2 s_1$ is resolvable. The minimal ambiguities in this example are \begin{equation} \label{eq:minimalS3} \{ s_1 s_1 s_1, s_2 s_2 s_2, s_1 s_1 s_2 s_1, s_1 s_2 s_1 s_1, s_1 s_2 s_1 s_2 s_1\}, \end{equation} all of which are resolvable.

The Bergman diamond lemma states that, if the set of monomials possesses a sufficiently nice partial order compatible with our relations, then the irreducible monomials form a basis if
and only if every minimal ambiguity is resolvable. This reduces the proof of linear independence to a finite (and in fact, very small) amount of work. For example, it gives a very
efficient proof that the PBW basis is linearly independent, see \cite[\S 3]{Bergman}.

\begin{remark} If you are defining an algebra by generators and relations, it is good practice to check whether the overlap ambiguities are resolvable, even if you do not plan to use
the Bergman diamond lemma. When I referee a paper, I typically check the first few overlap ambiguities; you might be surprised how many significant errors I have found this way. Be
responsible, check your ambiguities! \end{remark}

Unfortunately, the Bergman diamond lemma does not apply to the Coxeter presentation of $S_n$ for $n \ge 4$! There is no choice of direction on the relations which will lead to resolvable
ambiguities. A thorough discussion of what goes wrong can be found in \S\ref{sec:coxeterproblems}. Similarly, the Bergman diamond lemma does not apply to the Hecke algebra $\HB(\a,\b)$, or
the quiver Hecke algebras of Khovanov-Lauda-Rouquier, or their generalizations in the forthcoming work \cite{EGarcia1}.

It is the goal of this paper to provide a variant of the Bergman diamond lemma which does apply to these presentations. That is, it applies to \emph{Hecke-type presentations} of
algebras and algebroids. These are presentations of monoidal categories (and their endomorphism rings), generated by certain crossings and dots (the dots are analogous to a polynomial subalgebra), with relations like \eqref{subeq:coxeter} and \eqref{eq:NH}. See Definition \ref{def:HeckePart1} for a precise definition. The upshot is the following theorem, stated more precisely as Theorem \ref{thm:MDLforHecke}.

\begin{thm} Given a Hecke-type presentation of an algebra with $n$ strands, there is a basis of the form $\{f \un{w}\}$, where $f$ lives in (a basis for) the ``polynomial" subalgebra, and $\un{w}$ is a fixed reduced expression for $w \in S_n$, if and only if the following types of ambiguities are resolvable: 
\begin{equation} \label{eq:minimalHecke} \{ s_1 s_1 s_1, s_1 s_1 s_2 s_1, s_1 s_2 s_1 s_1, s_1 s_2 s_1 s_2 s_1, s_1 s_2 s_1 s_3 s_2 s_1, s_1 s_1 f, s_1 s_2 s_1 f, s_1 f g \}. \end{equation}
\end{thm}

\begin{remark} There are many ambiguities of each type listed above. For example, $s_1 s_2 s_1 s_1$ is a permutation on three strands, and in a Hecke-type algebra these strands may have
different colorings, so there is one such relation for each of the possible colorings of the strands. Nonetheless, checking these ambiguities is a finite and relatively small amount of
work, and this is precisely the least amount of work one needs to do. \end{remark}

To prove this theorem, we use Manin-Schechtmann theory (or a minor variant thereupon) in an essential way. Manin and Schechtmann \cite{ManSch} made a careful analysis of reduced
expressions in the symmetric group, and orientations on the reduced expression graph, which we extend to non-reduced expressions. We believe this has independent interest. See
\S\ref{sec:coxeterproblems} for more details.

Along the way we discuss some other variants of the Bergman diamond lemma, such as a version for monoidal presentations of monoidal categories, and a version over polynomial rings (or
other base rings) which need not lie in the center of the algebra.

{\bf Addendum}. This paper appeared concurrently with a related and independent paper by Dupont \cite{Dupont}, which uses techniques established a decade earlier by Guiraud-Malbos
\cite{GuiMal}\footnote{Indeed, the original plan was to publicize this work after \cite{EGarcia1} was complete, but the appearance of \cite{Dupont} prompted an early release and the addition of several addenda.} Guiraud and
Malbos studied rewriting theory (e.g. Bergman-style arguments) in higher categories, and developed a theory of ``derivations'' which, in some sense, replaces the partial order used
in Bergman's diamond lemma. Their main example in \cite[Chapter 5.4]{GuiMal} is indeed the Coxeter presentation of the symmetric group. Dupont uses this technology, and other tools
from rewriting theory dealing with non-termination (a failure of the DCC property discussed in \S\ref{sec:original}), to apply Bergman-style arguments to the 2-category which
categorifies the quantum group. This is more impressive than our results, which apply only to the quiver Hecke algebras which categorify the positive half of the quantum group.
Dupont's thesis introduced the work of one community of mathematicians (working on rewriting theory in higher categories) to another community of mathematicians (categorification
theorists, my milieu) which were mostly oblivious to it. There is a significant amount of recent literature on rewriting theory I was unaware of when I wrote this paper, but rather
than recalling it here, we point the reader to \cite{Dupont, GuiMal} for a more detailed account. Regardless, we feel that our methods (using Manin-Schechtmann theory) are
independent from and more explicit than those of Guiraud-Malbos, so are still a worthwhile addition to the literature.

{\bf Acknowledgements}. The idea for this paper germinated long ago, in response to a question of Mikhail Khovanov when I was his admiring graduate student. Thanks to Agustin Garcia for the impetus to release this upon the world. The author was supported by NSF CAREER grant DMS-1553032, and by the
Sloan Foundation. We would also like to thank the anonymous referee for helpful suggestions, which improved the clarity of the exposition.

\section{Bergman's diamond lemma}
\label{sec:original}

\subsection{Statement of the lemma}
\label{subsec:originaldiamond}

Fix a commutative ring $\Bbbk$, which will be the base ring for all our constructions.

Suppose one has defined an algebra $\AC$ over $\Bbbk$ by generators $\SC$ and relations $\RC$. Let $X$ denote the set of \emph{monomials}, i.e. words in $\SC$. Given words $A, B \in X$,
we let $AB$ denote their concatenation. Each relation $r \in \RC$ is a $\Bbbk$-linear combination inside $\Bbbk \cdot X$, and the ideal $I$ generated by these relations is the same as
the $\Bbbk$-submodule spanned by $ArB$ for each $A, B \in X$. So, as a $\Bbbk$-module, $\AC$ is the quotient of $\Bbbk \cdot X$ by $I$.

Let $\le$ be a partial order on $X$. We say that it is a \emph{semigroup partial order} if $A, B, B', C \in X$ and $B \le B'$ implies that $ABC \le AB'C$. We say it is \emph{compatible
with $\RC$} if every relation $r \in \RC$ can be rewritten as an equation $W_r = f_r$, where $W_r \in X$, and $f_r$ is a $\Bbbk$-linear combination of monomials $A$ satisfying $A < W_r$.
In other words, every linear combination in $r$ has a unique maximal monomial $W_r$, and its coefficient in $\Bbbk$ is invertible.

To apply the Bergman diamond lemma, we will need a partial order $\le$ which \begin{itemize} \item is a semigroup partial order, \item is compatible with $\RC$, and \item satisfies the
descending chain condition (DCC). \end{itemize} Let us fix such a partial order for the following discussion. Note that there is no requirement that the partial order respects the length
of a word, or that $B < ABC$.

\begin{example} We follow this example throughout the section. Suppose that $\SC = \{x,y,z\}$ and $\RC = \{yx-xy-1, zx-xz-2, zy-yz-3\}$. Then we can choose the \emph{degree lexicographic order}, the (total) order $<$ where
$A < B$ if the word $A$ is shorter than $B$, or if they have the same length and $A$ is smaller in the lexicographic order, with $x < y < z$. Degree lexicographic orders are always semigroup
partial orders, and this order satisfies the DCC. The order is compatible with $\RC$, as we may write our relations as $yx = xy + 1$, $zx = xz + 2$, and $zy = yz + 3$.
\end{example}

We can choose to apply relations only in one direction, replacing the monomial $W_r$ with the linear combination $f_r$. Instead of writing $W_r = f_r$ we write $W_r \mapsto f_r$.
Formally, for each $A, B \in X$ and $r \in \RC$, we define an \emph{elementary resolution} $\rho_{ArB}$ to be the $\Bbbk$-linear map on $\Bbbk \cdot X$ which sends $AW_r B$ to $A f_r B$,
and fixes every other monomial. A \emph{resolution} is a composition of elementary resolutions. We call a monomial \emph{irreducible} if it is fixed by every resolution, i.e. it has no
subwords of the form $W_r$ for any relation $r$. Let $X_{\irr}$ denote the set of irreducible monomials.

\begin{example} A monomial is irreducible if and only if it has the form $x^a y^b z^c$ for some $a, b, c \ge 0$. Here are two different resolutions of the monomial $zyx$:
\begin{subequations}
\begin{equation} \label{eq:zyx1} zyx \mapsto yzx + 3x \mapsto yxz + 2y + 3x \mapsto xyz + z + 2y + 3x, \end{equation}
\begin{equation} \label{eq:zyx2} zyx \mapsto zxy + z \mapsto xzy + 2y + z \mapsto xyz + 3x + 2y + z. \end{equation}
\end{subequations}
\end{example}

By our assumptions on the partial order, a finite number of elementary resolutions will take any monomial to a linear combination of irreducible monomials. Consequently, $X_{\irr}$ is a
spanning set for $\AC$ over $\Bbbk$. The question is whether it is a basis. After all, a given monomial could conceivably be resolved into a linear combination in $\Bbbk \langle X_{\irr}
\rangle$ in two different ways, and the difference between these resolutions would give a linear relation in $X_{\irr}$.

Suppose a monomial $A$ has both $W_r$ and $W_s$ appearing inside, for $r, s \in \RC$. We refer to the triple $(A,W_r,W_s)$ as an \emph{ambiguity}; here $W_r$ and $W_s$ refer to specific
subwords of $A$ (which may have many subwords of the form $W_r$), and we assume either that $W_r$ and $W_s$ are distinct subwords, or that $r \ne s$. In theory, $W_r = W_s$ for $r \ne s$
is possible, in which case $(W_r, W_r, W_s)$ is an ambiguity. There are two possible resolutions of $A$, one using the relation $r$ and the other $s$. We wish to show that these two
resolutions can be further resolved until they agree. More precisely, suppose that $A = B W_r C = D W_s E$. Then we say the ambiguity is \emph{(jointly) resolvable} if one can apply
further resolutions to $\rho_{BrC}(A)$ and $\rho_{DsE}(A)$ to arrive at the same element of $\Bbbk \cdot X$.

\begin{example} There is an ambiguity $(zyx, zy, yx)$, and our computations in \eqref{eq:zyx1} and \eqref{eq:zyx2} verify that this ambiguity is resolvable. \end{example}

The first main result of the diamond lemma says that if every ambiguity is resolvable, then there is a unique resolution of every monomial into a linear combination of irreducible
monomials.

When the subwords $W_r$ and $W_s$ do not overlap, we call $(A,W_r,W_s)$ a \emph{disjoint ambiguity}. These are the easy cases. Let $A = B W_r C W_s D$ for some (possibly empty) words $B,
C, D$. It is clear that one can resolve both $B f_r C W_s D$ and $B W_r C f_s D$ to the same linear combination $B f_r C f_s D$.

\begin{example} There is a disjoint ambiguity $(xxzyzy, zy, zy)$. Though unclear from the notation, the first $zy$ refers to the third and fourth letter, while the second $zy$ to the
fifth and sixth letter of the word $xxzyzy$. We do not feel like more precise notation is called for in this paper. \end{example}

There are only two ways in which subwords of a word can intersect nontrivially. The first is an \emph{overlap ambiguity}, where $W_r = B{\bf C}$ and $W_s = {\bf C}D$, and $A = LBCDM$, so that $W_r$ and $W_s$ intersect in the bold copies of $C$. Now it is not
obvious that $Lf_r DM$ and $LBf_s M$ should have a joint resolution. If the \emph{minimal} such ambiguity, namely $(BCD,W_r,W_s)$, is resolvable, then any overlap ambiguity containing
$BCD$ is resolvable.

The second kind of ambiguity is an \emph{inclusion ambiguity}, where $W_r = {\bf C}$ and $W_s = B{\bf C}D$. Again, if the minimal inclusion ambiguity $(BCD, W_r, W_s)$ is resolvable, then
any inclusion ambiguity containing $BCD$ is resolvable.

\begin{remark} \label{rmk:noinclusion} You have probably never seen a presentation with an inclusion ambiguity. One can always modify the presentation to avoid inclusion ambiguities, see
\cite[\S 5.1]{Bergman}. A very stupid example of a (non-resolvable) inclusion ambiguity would be the relations $x \mapsto 1$ and $x \mapsto 2$, and the triple $(x,x,x)$. \end{remark}

The second main result of the diamond lemma is merely the statement that one need only check the minimal overlap and inclusion ambiguities, to determine the resolvability of all
ambiguities. It is essential to observe that, while there are infinitely many overlap ambiguities, there are very few minimal overlap ambiguities, which reduces the
task of checking all overlap ambiguities to a finite amount of work.

\begin{example} Concluding our example from above. The ambiguity $(zyx, zy, yx)$ is the only minimal ambiguity, and it is resolvable. Hence $\{x^a y^b z^c\}_{a,b,c \ge 0}$ forms a basis for the algebra. \end{example}

\begin{example} \label{ex:S3} Consider again the Coxeter presentation for $S_3$, with $s = s_1$ and $t = s_2$. Let us use crossing diagrams to represent words in $s$ and $t$, as below. 
\[ {
\labellist
\small\hair 2pt
 \pinlabel {$s = $} [ ] at -10 10
 \pinlabel {$t = $} [ ] at 54 10
\endlabellist
\centering
\ig{1}{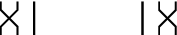}
} \]
Let us use the degree lexicographical order where $t < s$. Then we get elementary resolutions $sts \mapsto tst$, $ss \mapsto 1$, and $tt \mapsto 1$. Then the reader should confirm that the only minimal ambiguities are the overlap ambiguities pictured below, and that they are all resolvable.
\begin{equation} \ig{1}{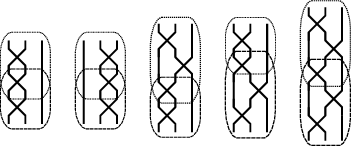} \end{equation}
Note that $tstt$ is not an overlap ambiguity, as one never uses the relation $sts \mapsto tst$ in the other direction. \end{example}

\begin{thm}[Bergman diamond lemma] \label{thm:BDL} (See \cite[Theorem 1.2]{Bergman}) Let $\AC$ be an algebra presented over a commutative (central) base ring $\Bbbk$ by generators $\SC$
and relations $\RC$. Suppose that one has a partial order $\le$ which is a semigroup partial order, is compatible with $\RC$, and satisfies the DCC. Then the irreducible monomials
$X_{\irr}$ form a basis for $\AC$ if and only if each ambiguity is resolvable, if and only if each minimal ambiguity (overlap and inclusion) is resolvable, if and only if each minimal
ambiguity is resolvable relative to $\le$ (see Definition \ref{def:relative} below). In this case, we say that the data $(\SC,\RC,\le)$ is \emph{Bergman type}. \end{thm}

\begin{example} \label{ex:S3badpresent} Consider a different presentation for the group algebra of $S_3$, with the same generators $s$ and $t$: the relations are $ss \mapsto 1$, $tt
\mapsto 1$, and $ststst \mapsto 1$. There are too many irreducible monomials (like $ststs$), and there are unresolvable ambiguities (like $s{\bf s}tstst$), so that the presentation is
not Bergman type. The Bergman diamond lemma is still true, but it is not particularly useful. Having a nice presentation is extremely important. \end{example}

\begin{remark} In Bergman's proof of Theorem \ref{thm:BDL}, overlap ambiguities and inclusion ambiguities are treated in precisely the same way. The kind of ambiguity does not affect the
essentially formal argument. \end{remark}

Bergman also introduces a more practical method to check whether ambiguities are resolvable: namely, being resolvable relative to $\le$.

\begin{defn} \label{def:relative} Let $A$ be a word. Define $I_{< A}$ as the $\Bbbk$-linear span of $\{B(W_r - f_r)C\}$ for all relations $r$ and words $B, C \in X$ such
that $B W_r C < A$. This $\Bbbk$-submodule encapsulates all the relations one has already deduced by induction, before examining the monomial $A$. Suppose that $A = B W_r C = D W_s E$ is
an ambiguity, and consider the two elementary resolutions $\rho_{BrC}(A)$ and $\rho_{DsE}(A)$. We say that the ambiguity $(A,W_r,W_s)$ is \emph{resolvable relative to $\le$} if
$\rho_{BrC}(A) - \rho_{DsE}(A) \in I_{<A}$. \end{defn}

To elaborate, suppose we have an ambiguity $A$ as in Definition \ref{def:relative}. To show that it is resolvable using the original definition, one would have to apply relations
$\RC$ \emph{only in the oriented direction} $W_r \mapsto f_r$ until $B f_r C$ and $D f_r E$ became equal. Using relative resolvability, one is permitted to apply the relations in
both directions, so long as the process stays below $A$ in the partial order. The ``restricted ideal'' $I_{<A}$ does not care about the directionality of relations! In practice,
the notion of relative resolvability allows one to take identities which one has already proven (possibly using relations in both directions), and use them to check whether
ambiguities are resolvable.

\subsection{A crucial warning}
\label{subsec:warning}

Now we must make the following warning, and make it very loudly!

The naive approach to using the Bergman diamond lemma is to do what we did in Example \ref{ex:S3}: choose for each relation an equation of the form $W_r = f_r$, and check the minimal
ambiguities. In checking the ambiguities, one never needs to use the partial order $\le$, so that it seems irrelevant. However, if one does not find an appropriate partial order $\le$,
then one can not apply the Bergman diamond lemma.

\begin{defn} \label{def:S3Redux} This definition (which we use later in the paper) uses the notation of Example \ref{ex:S3}. We define a partial order $\le_3$ on the set of words in
the alphabet $\{s,t\}$ as follows. \begin{itemize} \item If the word $A$ is shorter than the word $B$, then $A<_3 B$. \item If $A$ and $B$ have the same length, and they express
different elements of $S_3$, then they are incomparable. \item If $A$ and $B$ have the same length, and express the same element of $S_3$, but are not related by a sequence of braid
relations, then they are incomparable. \item Finally, if there is still a chance of them being comparable, then $A<_3 B$ if $A$ occurs after $B$ in lexicographic order. \end{itemize}
\end{defn}
	
Thus, for instance, $tst <_3 sts$. The reader should confirm that $\le_3$ is a partial order, and that it satisfies the hypotheses of Theorem \ref{thm:BDL}.

It is not obvious that a partial order satisfying the hypotheses of Theorem \ref{thm:BDL} exists, even when all the ambiguities are resolvable. A counterexample can be found in \cite[\S
5.4]{Bergman}. In \cite[\S 5.4]{Bergman} it is stated without proof that, if the base ring $\Bbbk$ is a domain, all ambiguities are resolvable, and no infinite chain of nontrivial
elementary resolutions can be applied to any monomial, then there exists a partial order satisfying the hypotheses of Theorem \ref{thm:BDL}. This last condition corresponds to $\le$
satisfying the DCC, and fails for many examples.

\begin{example} Suppose we wished to apply the relation $1 \mapsto ss$ in this direction instead. Then it is impossible for any compatible partial order to satisfy the DCC. \end{example}

\subsection{Comparison to Gr\"obner bases}
\label{subsec:vsGrobner}

When a presentation is not Bergman type, one can modify it until it is. There is a formal procedure to do this, see \cite[\S 5.3]{Bergman}. In Example \ref{ex:S3badpresent}, the
ambiguity $s{\bf s}tstst$ is not resolvable because $tstst$ and $s$ are irreducible. We can add a new relation $tstst \mapsto s$, which resolves the ambiguity $s{\bf s}tstst$, but
creates new ambiguities. Then we can force the unresolvable ambiguities to be resolvable by adding new relations, and can repeat. It is unclear if this process will terminate, but when
it does, one has a Bergman type presentation, and the monomials which remain irreducible at the end are a basis. This is effectively Buchberger's algorithm \cite{Buchberger} for finding
a Gr\"obner basis.

The theory of Gr\"obner bases, as initiated by Shirshov \cite{Shirshov}, is very similar to the Bergman diamond lemma. It fixes a set of generators, a partial order on monomials, and an
ideal $I$ (i.e. the ideal generated by the relations $\RC$), and seeks a \emph{Gr\"obner basis}, which is a basis for a generating subspace of $I$ (i.e. a collection of specific
relations) such that all ambiguities are resolvable. Unlike the Bergman diamond lemma, Gr\"obner basis theory cares about the ideal itself, not a particular set of generators for the
ideal (it goes on to find its own generators for the ideal).

We think of changing the relations as violating one goal of the diamond lemma: to take a fixed presentation and confirm a basis with the minimal amount of work. Instead, the procedure
above adds very many new relations, and does not represent a small amount of work. Admittedly, in the age of computers, Buchberger's algorithm can be extremely effective. Moreover, we
may be interested in whether a particular set of monomials, the irreducibles for our chosen set of relations, form a basis. For the Coxeter presentation of $S_n$, this is precisely
asking whether the basis of the group algebra is in bijection with $S_n$, rather than being some smaller set! For this question, adding more relations will change the set of irreducible
monomials, so it will not answer the question directly.

There is a large literature on Gr\"obner bases for various algebras, including the group algebras of Coxeter groups \cite{BokShi} (like the symmetric group) and their braid groups
\cite{BokutBraid}. It is acknowledged that (outside of dihedral type) the Coxeter presentation is not Bergman type, and (like this paper) that vast simplifications arise when one treats
the commuting relations \eqref{eq:sisj} in different fashion to the other relations. Instead of finding Gr\"obner bases via an algorithm, we do something different, modifying the
Bergman diamond lemma itself to let the Coxeter presentation of the symmetric group be Bergman type. Doing this requires an in-depth study of non-reduced expressions, which we have not
seen in the previous works we have explored.

\begin{remark} By no means am I an expert on the literature of Gr\"obner bases or the Bergman diamond lemma, so I would appreciate knowing of other attempts along these lines. \end{remark}

\subsection{The diamond lemma for linear categories}
\label{subsec:categorydiamond}

In this section and the next we describe some simple variants on the Bergman diamond lemma, which follow from it directly.

Let $\Bbbk$ be a commutative ring. A \emph{(free) $\Bbbk$-linear category} is a category where the morphism spaces are (free) $\Bbbk$-modules, and where composition is $\Bbbk$-bilinear.
Any $\Bbbk$-algebra can be viewed as a $\Bbbk$-linear category with one object.

Given a $\Bbbk$-linear category $\CC$, one can consider the (possibly infinite) direct sum $\AC = \bigoplus_{n,m \in \Ob(\CC)} \Hom(n,m)$. This is not necessarily an algebra, as when
$\Ob(\CC)$ is infinite it will not have a unit. Instead, it is an \emph{algebroid} or \emph{locally unital algebra}. It has a collection $\{\id_n\}_{n \in \Ob(\CC)}$ of orthogonal
idempotents, whose sum $\sum \id_n$ need not exist in $\AC$ (when the sum is infinite), but functions like the identity element. That is, this (infinite) sum acts via a finite
sum when multiplying by any element of $\AC$, and this finite sum acts by the identity.

Almost every result in ring theory has an immediate and obvious analog for algebroids, though they are not always written down in the literature. For example, it is easy to adapt
Bergman's diamond lemma to $\Bbbk$-linear categories, constructed by generators and relations. Quiver algebras and their quotients are excellent examples to keep in mind, where the generators are thought of as arrows between two objects (i.e. vertices), words are paths, and relations are linear combinations of paths set to zero.

The main difference between algebras and algebroids is that every morphism has a source and a target: one can separate the set $X$ of monomials into sets $X(n,m)$ of monomials which are
morphisms from $n$ to $m$. One may as well assume that the partial order $\le$ is defined separately on each set $X(n,m)$ (or equivalently, that morphisms with different sources or
targets are incomparable).

\begin{example} \label{ex:Scat} Consider the \emph{symmetric category} $\SymCat$. It has one object for each $n \in \NM$, and $\Hom(n,m) = 0$ if $n \ne m$. One has $\End(n) = \Bbbk[S_n]$.
Then $\SymCat$ is a $\Bbbk$-linear category. The presentations of each symmetric group individually give rise to a presentation of $\SymCat$. \end{example}

\subsection{When the base ring is not central}
\label{subsec:noncentral}

Often one has an inclusion of rings $\Bbbk \subset \AC$, where $\AC$ is free over $\Bbbk$ as a left module (or a right module), but $\Bbbk$ does not live in the center of $\AC$. The inclusion of the polynomial ring inside the nilHecke algebra (or more generally, a quiver Hecke algebra) is an example of this phenomenon. One can
ask whether the Bergman diamond lemma can be used to prove that $\AC$ is free over $\Bbbk$ as a left module. We assume $\Bbbk$ is commutative, for simplicity.

When presenting $\AC$ over $\Bbbk$ with additional generators $\SC$, monomials in $\SC$ are not sufficient to describe all elements of $\AC$, as we may need to multiply by $\Bbbk$ as
well. That is, we must consider \emph{$\Bbbk$-words}: sequences of the form $AfBgC\cdots$ where $A, B, C, \ldots$ are words in $\SC$, and $f, g, \ldots \in \Bbbk$. The relations in $\RC$
will involve $\Bbbk$-words. We refer to words just in the alphabet $\SC$ as \emph{true words}. More precisely, we are interested in whether $\AC$ is free over $\Bbbk$ as a left module,
with a basis consisting entirely of true words.

If this is to happen, then every $\Bbbk$-word must be writable (after applying relations) as a left $\Bbbk$-linear combination of true words. In particular, for each $s \in \SC$ and $f
\in \Bbbk$, there must be an equality $sf = \sum g_i A_i$ for true words $A_i$ and coefficients $g_i \in \Bbbk$. However, the $\Bbbk$-word $sf$ does not have many subwords, so for it to
be non-irreducible in a Bergman-type presentation, there must be a relation of the form $sf = \sum g_i A_i$, for every $s \in \SC$, and at least for generators $f$ of $\Bbbk$. We will
return to this point shortly.

One approach is to find a presentation for $\Bbbk$ over a subring $\Bbbk'$ which is central in $\AC$ (like the ground field), and to study the ``corresponding presentation" of $\AC$ over
$\Bbbk'$.\footnote{There are many such presentations: each $f, g, \ldots$ in a $\Bbbk$-word in a relation in $\RC$ would have to be expressed using linear combinations of words in the generators
of $\Bbbk$, and this expression is not unique.} This is an effective approach when, for example, $\Bbbk$ is a polynomial ring. This is not the approach we are interested in, as we would
rather treat the ring $\Bbbk$ as a black box.

Our approach is to add to $\SC$ a new generator $m_f$ (multiplication by $f$) for each $f \in \Bbbk$. Yes, we mean {\bf every} element of $\Bbbk$, not just a basis.
Every instance of multiplication by $f$ in a $\Bbbk$-word in a relation $r \in \RC$ we replace by $m_f$.\footnote{Unlike the previous footnote, this operation is unique.}

\begin{example} Relation \eqref{eq:NH} in the nilHecke algebra can be written for general polynomials as \begin{equation} \label{eq:NHbetter} s_i m_f = m_{\phi(f)} s_i + m_{\pa(f)} \id,
\end{equation} where $\phi$ and $\pa$ are certain maps $\Bbbk \to \Bbbk$. Then \eqref{eq:NHbetter} can be thought of as a family of relations, one for each $f \in \Bbbk$. \end{example}

\begin{remark} If the reader prefers, when working over a field, one can also choose a basis of $\Bbbk$ over this field, and only define symbols $m_f$ for $f$ inside this basis. One
obtains an equivalent theory, requiring a similar amount of work to resolve. One philosophical disadvantage is that one loses the wholistic approach to relations like
\eqref{eq:NHbetter}, which are encoded in terms of maps $\phi$ and $\pa$ instead of particular linear combinations for each basis element. \end{remark}

We need to add some new relations in order for this new presentation to have the desired properties. One relation we add is \begin{equation} \label{eq:mult} m_f m_g \mapsto m_{fg}
\end{equation} for each $f, g \in \Bbbk$. Now, so long as there is a relation \begin{equation} \label{eq:sf} s m_f \mapsto \sum m_{g_i} A_i \end{equation} for each $s \in \SC$ and $f
\in \Bbbk$, only words of the form $m_f X$ for $X$ a true word could possibly be irreducible. We assume that such a family of relations exists.

\begin{remark} \label{rmk:partialwithm} One also needs to choose a partial order on monomials compatible with \eqref{eq:sf}. When using a degree lexicographic order or similar order for the generators $\SC$, one can extend this to a partial order by declaring that all symbols $m_f$ are incomparable with each other, and lexicographically smaller than all symbols in $\SC$. This results in a partial order where $B m_f C < A$ implies that $B m_g C < A$ for all $f, g \in \Bbbk$. All polynomials are treated equally. \end{remark}

Now we need to impose a ``relation'' which takes $m_f$ appearing on the left of a word, and replaces it with multiplication by $f$ on the remainder of the word. Let us introduce new
notation for words, writing $\bullet$ for the empty word, and placing $\bullet$ at the beginning of every word. We now introduce the ``left relation'' \begin{equation} \label{eq:bullet}
\bullet m_f \mapsto f \bullet, \end{equation} for each $f \in \Bbbk$, where the left hand side is a word, and the right hand side is a left $\Bbbk$-linear combination. This ``left
relation'' is only permitted to be applied at the beginning of a word. With \eqref{eq:bullet} in place, only true words can be irreducible.

\begin{remark} It is easy to modify the Bergman diamond lemma to allow for left relations. In fact, this is essentially the Bergman diamond lemma for left modules over a ring; we are viewing the ring $\AC$ as a left $\Bbbk$-module. \end{remark}

\begin{remark} The relation $m_f + m_g = m_{f+g}$ is not necessary given the relation \eqref{eq:bullet}; this relationship between elements of $\Bbbk$ is supposed to be handled by
taking $\Bbbk$-linear combinations of words. It is also not desireable to have $m_f + m_g = m_{f+g}$ as a relation, as it would make the existence of a compatible partial order
$\le$ effectively impossible. However, one is permitted to use this relation as desired when checking the ambiguities, because one can check ambiguities relative to $\le$, see
Definition \ref{def:relative} and following. When resolving an ambiguity $A$, if $B m_f C < A$ then (for reasonable partial orders, see Remark \ref{rmk:partialwithm}) also $B m_g C < A$ and $B (m_{f+g}) C < A$, thus
\[ B m_f C + B m_g C - B m_{f+g} C \in I_{< A}. \]
\end{remark}

Now (assuming the existence of a suitable partial order) we can apply the original Bergman diamond lemma to this modified presentation of $\AC$. Adding the relations \eqref{eq:bullet}
and \eqref{eq:mult} creates (at least) two new overlap ambiguities. One is $(m_f m_g m_h, m_f {\bf m_g}, {\bf m_g} m_h)$, which is always resolvable by associativity in $\Bbbk$. The
other is $(\bullet m_f m_g, \bullet {\bf m_f}, {\bf m_f} m_g)$, which is also always resolvable. Note that the resolvability of the family of ambiguities $(s m_f m_g, s {\bf m_f}, {\bf
m_f} m_g)$ is an interesting and nontrivial statement about the family of relations \eqref{eq:sf}.

\section{The problem with the Coxeter presentation}
\label{sec:coxeterproblems}

\subsection{Depiction of the problem}
\label{subsec:coxeterproblem}

The Coxeter presentation of the symmetric group $S_n$ was given in the introduction as \eqref{subeq:coxeter}. We have already seen in Example \ref{ex:S3} and \S\ref{subsec:warning} that
this presentation is Bergman-type when $n=3$. However, it is not Bergman-type for $n \ge 4$.

Let $s = s_1$, $t = s_2$, and $u = s_3$ inside $S_4$. The quadratic relations must have the form $ss \mapsto 1$, $tt \mapsto 1$, $uu \mapsto 1$, as if any of these elementary
resolutions goes the other way, no compatible partial order would satisfy the DCC. Suppose, for example, that the remaining relations are $sts \mapsto tst$, $tut \mapsto utu$, and $su
\mapsto us$. Then $(stsu, st{\bf s}, {\bf s}u)$ is an unresolvable overlap ambiguity, as both $tstu$ and $stus$ are irreducible. The longest element of $S_4$ has multiple irreducible
reduced expressions, namely $utustu$ and $tustus$ and $ustust$. The reader can confirm that no other way to ``orient'' the relations (i.e. replacing $sts\mapsto tst$ with $tst \mapsto
sts$) will have resolvable ambiguities either.

Let us rephrase the problem. For each element $w \in S_n$, let $\G_w$ denote its \emph{reduced expression graph}. This is the graph whose vertices are reduced expressions for $w$, and
the edges are applications of a single braid relation, either \eqref{eq:sisj} or \eqref{eq:sisjsi}. The edges can be labeled by the pair $\{i,j\}$ of indices which are involved in the
braid relation; there are ${n-1}\choose{2}$ possible labelings. If one chooses a partial order compatible with the relations, then for each pair $\{i,j\}$ with $j \ne i \pm 1$, either
$s_i s_j$ is bigger than $s_j s_i$, or the other way around. This gives an orientation on the reduced expression graph of $s_i s_j$ (the arrow goes from the larger word to the smaller one). If the partial order is a semigroup partial order,
then this same orientation applies to every edge labeled $\{i,j\}$ in any reduced expression graph. Similar considerations can be made for pairs $\{i,j\}$ with $j = i \pm 1$. There are
$2^{{n-1}\choose{2}}$ orientations on reduced expression graphs, obtained by choosing an orientation for each pair $\{i,j\}$.

\begin{remark} The reader who has never encountered reduced expression graphs before will be greatly served by drawing the reduced expression graph of the longest element of $S_4$. There
are 16 reduced expressions; pick one and then start applying braid relations to find the rest. The answer can be found, for example, in \cite[\S 3.1]{EThick}. \end{remark}

A sink in the oriented reduced expression graph is an irreducible monomial. Thus, if the presentation is Bergman-type, then there must be an orientation where each reduced expression
graph has a unique sink. As noted, the oriented reduced expression graph of the longest element of $S_4$ always has multiple sinks.

\subsection{Avoidance of the problem}
\label{subsec:coxeteravoid}

Now we point out the real mathematical meat of this paper: the only problem is choosing an orientation for the relation \eqref{eq:sisj}, i.e. for pairs $\{i,j\}$ with $j \ne i \pm 1$.

Let $\bar{\G}_w$ denote the quotient of $\G_w$ by all edges with a label $\{i,j\}$ for $j \ne i \pm 1$. That is, a vertex of $\bar{\G}_w$ is an equivalence class $[\un{w}]$ of
reduced expressions of $w$, where reduced expressions $\un{w}$ and $\un{w}'$ are \emph{equivalent} if they differ only by commuting relations like $su = us$. There is an edge from
one vertex to another, labeled by $\{i,j\}$ with $j = i \pm 1$, if there is a representative of each of the two equivalence classes which is related by \eqref{eq:sisjsi}.

The following is a theorem of Manin and Schechtmann, and (in this particular form) is implicit but not explicitly stated in \cite{ManSch}. For this reformulation, see \cite[\S 1.3, \S
3.1]{EThick}.

\begin{prop} \label{prop:MS} Place the antilexicographic orientation on $\bar{\G}_w$, where each edge is oriented $s_i s_{i+1} s_i \mapsto s_{i+1} s_i s_{i+1}$. Then $\bar{\G}_w$ has a
unique source and a unique sink. In fact, $\bar{\G}_w$ is a \emph{graded graph}: every vertex $[\un{w}]$ has a height $h([\un{w}]) \in \NM$ and oriented edges increase the height by one.
\end{prop}

We recall Manin-Schechtmann theory and prove this proposition in \S\ref{subsec:inversions}, see Corollary \ref{cor:MS}.

\begin{remark} When $w$ is the longest element of $S_n$, $\bar{\G}_w$ is the Hasse diagram of the second \emph{higher Bruhat order}. In \cite{ManSch} they define a family of higher
Bruhat orders, starting with the familiar Bruhat order, then $\bar{\G}_w$, and beyond. Each higher Bruhat order is a graded graph with unique source and sink, whose vertices are
equivalence classes of paths from source to sink in the previous Bruhat order. \end{remark}

In particular, this indicates that, if there were some way to apply a modified Bergman diamond lemma where one was allowed to apply the relation $s_i s_j = s_j s_i$
in either direction, then there might be a chance that the Coxeter presentation is ``modified Bergman-type." This is what we accomplish in this paper.

\subsection{Manin-Schechtmann theory: triples and inversions}
\label{subsec:inversions}

The results in this section are due to Manin-Schechtmann. More precisely, they all effectively come from the $k=1$ and $k=2$ cases of \cite[\S 2, Theorem 3]{ManSch} and \cite[\S 2, Lemma 8]{ManSch}, though this should not be obvious to the reader at first glance. Because we use crossing diagrams to represent permutations, we call indices $1 \le i \le n$ by the name
\emph{strands}.

\begin{defn} Given a permutation $w \in S_n$, and \emph{inversion of $w$} is a pair of strands $(i|j)$ with $w(j) > w(i)$. We always write pairs $(i|j)$ in order, with $i < j$. We write $I(w)$ for the \emph{inversion set} of $w$, the set
of all its inversions. \end{defn}

Note that the length of $w$ is the size of the inversion set. For example, the inversion set of the longest element includes all ${n}\choose{2}$ pairs $(i|j)$.

\begin{defn} To any given crossing in an expression (not necessarily reduced), we associate the pair $(i|j)$ of strands which is crossed. The following example, where each crossing is
labeled with its pair, should make this definition clear. Note that we read a diagram from bottom to top. \end{defn}
\begin{equation} \label{inversionexample} {
\labellist
\tiny\hair 2pt
 \pinlabel {\small $1$} [ ] at 27 5
 \pinlabel {\small $2$} [ ] at 36 5
 \pinlabel {\small $3$} [ ] at 45 5
 \pinlabel {$(1|2)$} [ ] at 18 17
 \pinlabel {$(1|3)$} [ ] at 18 29
 \pinlabel {$(2|3)$} [ ] at 18 41
 \pinlabel {$(1|2)$} [ ] at 18 53
 \pinlabel {$(1|3)$} [ ] at 18 65
\endlabellist
\centering
\ig{1.3}{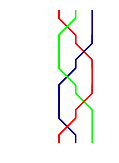}
}\end{equation}

The string diagram from \eqref{inversionexample} corresponds to the expression $s_1 s_2 s_1 s_2 s_1$, or $ststs$ where $s = s_1$ and $t = s_2$ as before. Because diagrams are read from bottom to top, and permutations act from right to left, it is the leftmost $s_1$ which corresponds to the top crossing labeled $(1|3)$, and the rightmost $s_1$ which corresponds to the bottom crossing labeled $(1|2)$.

Let $\un{w}$ be a reduced expression for an element $w \in S_n$. Since the length of $\un{w}$ agrees with the size of $I(w)$, and since each crossing adds at most one inversion, it must be the case that every pair $(i|j) \in I(w)$ labels precisely one crossing in $\un{w}$.

\begin{defn} Let $\un{w}$ be a reduced expression for $w \in S_n$. To this reduced expression we associate a total order $\prec$ on $I(w)$ where $(i|j) \prec (k|l)$ if the crossing labeled $(i|j)$ appears below $(k|l)$. We call this the \emph{crossing sequence} of the reduced expression. This map from reduced expressions to crossing sequences is injective. \end{defn}
	
In contrast to the notation $(i|j) \prec (k,l)$, which depends on a choice of reduced expression $\un{w}$, we will continue to use $(i|j) < (k|l)$ to indicate the usual lexicographic order. Manin-Schechtmann theory comes in when one compares the two orders $\prec$ and $<$.

\begin{example} In the reduced expression $sts$ in $S_3$, the rightmost $s$ crosses $(1|2)$, then $t$ crosses $(1|3)$, and finally $s$ crosses $(2|3)$. This yields the
lexicographic order on $I(sts)$: $((1|2) \prec (1|3) \prec (2|3))$. Meanwhile, the reduced expression $tst$ yields the antilexicographic order $((2|3) \prec (1|3) \prec (1|2))$.
\[{
\labellist
\small\hair 2pt
 \pinlabel {$(1|2)$} [ ] at 7 12
 \pinlabel {$(1|3)$} [ ] at 7 25
 \pinlabel {$(2|3)$} [ ] at 7 38
 \pinlabel {$(2|3)$} [ ] at 61 12
 \pinlabel {$(1|3)$} [ ] at 61 25
 \pinlabel {$(1|2)$} [ ] at 61 38
\endlabellist
\centering
\ig{1.3}{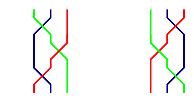}
}\]
\end{example}

Not every set of pairs of strands is the inversion set of a permutation. For example, one can not cross $(1|3)$ before getting the strand $2$ out of the way somehow. Moreover, not every
total order on an inversion set corresponds to a reduced expression: there are six orders on $\{(1|2), (1|3), (2|3)\}$, but only two correspond to reduced expressions.

\begin{defn} A \emph{triple} of strands is $(i|j|k)$ with $1 \le i<j<k \le n$. The \emph{packet} $P(i|j|k)$ of a triple is the set $((i|j) < (i|k) < (j|k))$, given its lexicographic
ordering. Let $I$ be a set of pairs of strands $\{(i|j)\}$ (for example, $I$ might be an inversion set). We call a triple $(i|j|k)$ a \emph{prefix triple} for $I$ if $P(i|j|k) \cap I$ is a prefix of $P(i|j|k)$. That is, the intersection is either the empty set, $\{(i|j)\}$, $\{(i|j),(i|k)\}$, or all of $P(i|j|k)$. We call $(i|j|k)$ a \emph{suffix triple} for $I$ if $P(i|j|k) \cap I$ is a suffix of $P(i|j|k)$. We call $(i|j|k)$ a \emph{full triple} for $I$ if $P(i|j|k) \subset I$. In particular, a full triple is both prefix and suffix. When $I = I(w)$, we say instead that $(i|j|k)$ is a \emph{prefix triple for $w$}, etcetera. \end{defn}

\begin{example} For the set $I = \{(1|3)\}$, the triple $(1|2|3)$ is neither prefix nor suffix. \end{example}
	
\begin{example} For $i<j<k$, the triple $(i|j|k)$ is full for $w$ if and only if $w(k)<w(j)<w(i)$. \end{example}
	
The key point to make about triples is that, whenever a subword $s_i s_{i+1} s_i$ or $s_{i+1} s_i s_{i+1}$ appears inside an expression, it is because the packet of some triple
$(i|j|k)$ appears in consecutive order inside the crossing sequence of the expression. Applying the braid relation $s_i s_{i+1} s_i \mapsto s_{i+1} s_i s_{i+1}$ to a reduced expression
will flip the order induced on this packet from lexicographic to antilexicographic. On the other hand, applying a braid relation $s_i s_j = s_j s_i$ for $j \ne i+1$ will flip the order
on two adjacent pairs $(i|j)$ and $(k|l)$ with $\{i,j\} \cap \{k,l\} = \emptyset$. In particular, $(i|j)$ and $(k|l)$ appear in no common packets, so that swapping their order is not
detected by any triple.

We now state a number of theorems, and we strongly encourage the reader to test out these theorems on (all) the reduced expressions of the longest element of $S_4$.

\begin{thm} \label{thm:MS} \begin{enumerate}
\item A set $I$ of pairs $\{(i|j)\}$ is the inversion set $I(w)$ of some permutation $w \in S_n$ if and only if each triple $(i|j|k)$ is either a prefix or a suffix triple for $I$.
\item A total order $\prec$ on the inversion set $I(w)$ corresponds to some reduced expression $\un{w}$ for $w$ if and only if,
for each triple $(i|j|k)$: \begin{itemize} \item if $(i|j|k)$ is prefix for $w$, but not full, then the ordering induced on $I(w) \cap P(i|j|k)$ by $\prec$ is
the lexicographic ordering; \item if $(i|j|k)$ is suffix for $w$, but not full, then the ordering induced on $I(w) \cap P(i|j|k)$ by $\prec$ is the
antlexicographic ordering; \item if $(i|j|k)$ is full for $w$, then the ordering induced on $P(i|j|k) \subset I(w)$ by $\prec$
is either the lexicographic or the antilexicographic ordering. \end{itemize} \end{enumerate}
\end{thm}

\begin{defn} For a reduced expression $\un{w}$ of $w$, define its \emph{higher inversion set} $J(\un{w})$ to be the set of full triples $(i|j|k)$ for $w$ such that $\prec$ induces the antilexicographic order on $P(i|j|k)$. \end{defn}

\begin{thm} \label{thm:MSagain} \begin{enumerate} \item If $\un{w} \mapsto \un{w}'$ are reduced expressions related by a single braid relation $s_i s_{i+1} s_i \mapsto s_{i+1} s_i s_{i+1}$, then $J(\un{w})
\subset J(\un{w}')$, and the only triple in $J(\un{w}')$ not in $J(\un{w})$ is the triple $(i|j|k)$ of strands involved in this braid relation. Consequently, applying a braid relation
is sometimes called a \emph{packet flip}, or \emph{flipping the packet $P(i|j|k)$}. \item For reduced expressions $\un{w}$ and $\un{w}'$ of $w \in S_n$, one has $J(\un{w}) = J(\un{w}')$
if and only if $\un{w}$ and $\un{w}'$ are equivalent. \end{enumerate}
\end{thm}

For part (b) above and the results below, we recall that two reduced expressions are \emph{equivalent} if they are related by commuting braid relations $s_i s_j = s_j s_i$ for $j \ne i \pm 1$, and that $\bar{\G}_w$ is the graph of equivalence classes of reduced expressions.

\begin{remark} Manin and Schechtmann also have a result indicating which sets of triples $J$ are the higher inversion sets of a reduced expression, by examining packets of \emph{quadruples}, analogous to part (1) of Theorem \ref{thm:MS}. This question relates to the third higher Bruhat order. \end{remark}

Finally we have the result which proves Proposition \ref{prop:MS}.

\begin{prop} \label{prop:MS2} Suppose that $\un{w}$ is a reduced expression which has at least one full triple not in $J(\un{w})$. Then there is some full triple $(i|j|k) \notin
J(\un{w})$ whose packet is \emph{flippable}, i.e. there exists a reduced expression $\un{w}'$ equivalent to $\un{w}$ in $\bar{\G}_w$, and a braid relation $s_\ell s_{\ell+1} s_\ell
\mapsto s_{\ell+1} s_\ell s_{\ell+1}$ which can be applied to $\un{w}'$ which adds $(i|j|k)$ to the higher inversion set. \end{prop}

\begin{cor} \label{cor:MS} The function $h(\un{w})$ given by the size of $J(\un{w})$ is invariant on equivalence classes of reduced expressions, and descends to a height function on $\bar{\G}_w$. There is a unique sink in $\bar{\G}_w$, whose higher inversion set is the set of full triples for $w$. (There is also a unique source whose higher inversion set is empty, which follows from an ``backwards'' version of Proposition \ref{prop:MS2}.) \end{cor}
	
The paper \cite{ManSch} only proves Proposition \ref{prop:MS2} for the longest element of $S_n$, so for completeness we include a proof for an arbitrary permutation $w$. Our proof is different than the one in \cite{ManSch}, though a proof along their lines would also work.

\begin{proof} Let $\un{w} = ((i_1|j_1) \prec (i_2|j_2) \prec \cdots \prec (i_d|j_d))$ be a reduced expression for $w$, for which we record its crossing sequence. Amongst the set of full triples $(i|j|k)$ for which $\un{w}$ induces the lexicographic order on the packet $P(i|j|k)$, choose one where the subword starting with $(i|j)$ (then passing through $(i|k)$) and ending with $(j|k)$ has minimal length. We will use induction on this minimal length to prove the result. If this length is $3$, i.e. the packet appears consecutively in $\un{w}$, then the packet can obviously be flipped. If not, then some other pair $(x|y)$ intervenes, either between $(i|j)$ and $(i|k)$, or between $(i|k)$ and $(j|k)$, or both. We will assume an intervening pair between $(i|j)$ and $(i|k)$, as the other case is treated similarly.
	
Suppose that \[ \un{w} = ( \cdots \prec (i|j) \prec (x|y) \prec \cdots),\] with $(x|y) \ne (i|k)$. If $\{i,j\} \cap \{x,y\} = \emptyset$ then a commuting relation will reverse the
order, yielding \[ \un{w}' = (\cdots \prec (x|y) \prec (i|j) \prec \cdots). \] So $\un{w}'$ is equivalent to $\un{w}$ and the minimal length in $\un{w}'$ is shorter. If $\{i,j\} \cap
\{x,y\}$ is nonempty, then there are several cases.

If $i = y$ so that $x < i < j < k$, then consider the triple $(x|i|j)$. The pairs $(i|j)$ and $(x|i)$ appear consecutively in this order in $\un{w}$, but $((x|i) \prec (i|j))$ does not appear in
either the lexicographic or the antilexicographic order on $P(x|i|j)$, a contradiction.

If $j = x$ so that $i < j < y$, then consider the triple $(i|j|y)$. The pairs $(i|j)$ and $(j|y)$ appear consecutively, which is again a contradiction.

If $j = y$ so that $x < j < k$, then consider the triple $(x|j|k)$. The pairs $(x|j)$ and $(j|k)$ appear in this order (though not consecutively). This means that $(x|k)$ must be somewhere in between, and hence $(x|j|k)$ is a full triple, appearing in lexicographic order. But then this triple $(x|j|k)$ has shorter length than $(i|j|k)$, so induction implies that some packet can be flipped.

Finally, if $i = x$, then there are several cases depending on where $y$ falls. Recall that \[\un{w} = (\cdots \prec (i|j) \prec (i|y) \prec \cdots \prec (i|k) \prec \cdots \prec (j|k) \prec \cdots).\]

If $i<y<j<k$ then consider the triple $(i|y|j)$. Since $(i|j)$ appears before $(i|y)$, the pair $(y|j)$ is also in $I(w)$ and occurs somewhere before $(i|j)$. Now consider the triple $(y|j|k)$. Since $(y|j)$ and $(j|k)$ both appear, so must $(y|k)$, at some point before $(j|k)$. Finally, consider the triple $(i|y|k)$. It is a full triple, and $(i|y)$ appears before $(i|k)$ so it must be in lexicographic order. Then the full triple $(i|y|k)$ has shorter length than $(i|j|k)$.

If $i<j<y<k$ then consider the triple $(j|y|k)$. Since $(j|k)$ does appear, either $(j|y)$ or $(y|k)$ must appear before it. If $(j|y)$ appears then the triple $(i|j|y)$ is full, with $(i|j)$ appearing before $(i|y)$, so it is lexicographic, and has shorter length. If $(y|k)$ appears then the triple $(i|y|k)$ is full, with $(i|y)$ appearing before $(i|k)$, so it is lexicographic, and has shorter length.

If $i<j<k<y$ then consider the triple $(i|k|y)$. Since $(i|y)$ appears before $(i|k)$, the $(k|y)$ is also in $I(w)$ and occurs somewhere before $(i|y)$. Now consider the triple $(j|k|y)$. Since $(k|y)$ and $(j|k)$ both appear, so must $(j|y)$, at some point before $(j|k)$. Finally, consider the triple $(i|j|y)$. It is a full triple, and $(i|j)$ appears before $(i|y)$ so it must be in lexicographic order. Then the full triple $(i|j|y)$ has shorter length than $(i|j|k)$.

This concludes the enumeration of cases, and the proof that some flippable packet exists.
\end{proof}

\subsection{Non-reduced expressions}
\label{subsec:partial}

For each $w \in S_n$, let $X_w$ denote the (infinite) set of all expressions for $w$, not necessarily reduced. Let $\bar{X}_w$ denote the set of equivalence classes in $X_w$,
identifying $\un{w}$ and $\un{w}'$ if they differ only by commuting relations \eqref{eq:sisj}. Then $X_w$ and $\bar{X}_w$ can be given the structure of graphs, analogous to $\G_w$ and
$\bar{\G}_w$, where one also has an edge for each relation $s_i s_i \mapsto 1$. In fact, $\G_w$ will be a subgraph of $X_w$, and $\bar{\G}_w$ of $\bar{X}_w$.

There is an additional structure on the graph of all expressions, which is not present when one restricts to reduced expressions. Namely, if $\un{w} \in X_w$ and $\un{y} \in X_y$ then
the concatenation $\un{w}\un{y}$ lives in $X_{wy}$. This concatenation operation descends to equivalence classes \[ \bar{X}_w \times \bar{X}_y \to \bar{X}_{wy}.\]

There is one more operation, which we could have discussed previously for the graphs $\G_w$ and $\bar{\G}_w$. For each $1 \le n \le m$ and $0 \le k \le m-n$ there is an inclusion $\iota
= \iota_k^{n \to m} \co S_n \to S_m$, via the maps $S_n \into S_k \times S_n \times S_{m-n-k} \into S_{m}$. That is, $\iota$ sends $s_i$ to $s_{i+k}$ for each $1 \le i \le n-1$. This
map $\iota$ induces inclusions of graphs $X_w \to X_{\iota(w)}$, which also descends to equivalence classes.

\begin{defn} \label{def:orientXw} Place an orientation on $\bar{X}_w$ for each $w$, where the edges are oriented $s_i s_{i+1} s_i \mapsto s_{i+1} s_i s_{i+1}$ and $s_i s_i \mapsto 1$. \end{defn}

This is the only reasonable orientation extending Manin-Schechtmann's orientation on $\bar{\G}_w$. We will prove later in this section that $\bar{X}_w$ has a unique sink.

Note that this orientation gives $\bar{X}_w$ the structure of a poset, where a sink is a minimal element. We now define a preorder on $X_w$, which will descend to a partial order
on $\bar{X}_w$, and which is stronger than the one induced by the orientation. This partial order will be one of the inputs to our generalized Bergman diamond lemma for Hecke-type
algebras. Recall the partial order $\le_3$ on words in $S_3$, built from the generators $s$ and $t$, given in Definition \ref{def:S3Redux}.

\begin{notation} If $\un{w}$ is an expression for $w \in S_n$, then for any pair $(i|j)$ of strands, let $n_{ij}$ denote the number of $(i|j)$ crossings in $\un{w}$. Clearly $\sum_{1
\le i < j \le n} n_{ij}$ is the length of $\un{w}$. If $\un{w}'$ is another expression, we write $n'_{ij}$ for the same statistic. \end{notation}

\begin{lemma} If $\un{w}$ and $\un{w}'$ are words in $S_3$ and have the same length, and they are comparable in $\le_3$, then $n_{12} = n'_{12}$, $n_{13} = n'_{13}$, and $n_{23} = n'_{23}$. \end{lemma}

\begin{proof} This is obvious, as braid relations never change the number of crossings of each type $(i|j)$. \end{proof}

\begin{defn} \label{def:partial} If $\un{w}$ is an expression and $(i|j|k)$ is a triple, let $\un{w}_{ijk}$ denote the expression inside $S_3$ obtained by tracing out the triple of strands $(i|j|k)$. An example is pictured below inside $S_8$, where the irrelevant strands are drawn without color:
\[{
\labellist
\small\hair 2pt
 \pinlabel {$2$} [ ] at 24 4
 \pinlabel {$4$} [ ] at 32 4
 \pinlabel {$6$} [ ] at 40 4
 \pinlabel {$\un{w} = $} [ ] at 6 44
 \pinlabel {$\un{w}_{246} = $} [ ] at 72 44
 \pinlabel {$ = ststs$} [ ] at 125 44
\endlabellist
\centering
\ig{1.3}{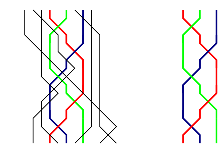}
} \]
For two expressions $\un{w}$ and $\un{w}'$ in $X_w$, we say that $\un{w} < \un{w}'$ if $\un{w}$ is shorter than $\un{w}'$, and that $\un{w} \le \un{w}'$ if they have the same length and if $\un{w}_{ijk} \le_3 \un{w}'_{ijk}$ for all triples $(i|j|k)$. \end{defn}

Clearly $\le$ is a transitive and reflexive on $X_w$, so it is a preorder.

\begin{remark} \label{rmk:partialextended} If desired, one can also extend this relation to the entire set $X$ of all expressions for any element. If $\un{w}$ and $\un{y}$ are expressions for different elements, then $\un{w} < \un{y}$ if and only if $\un{w}$ is shorter. In this section, we consider each $X_w$ separately. \end{remark}

\begin{lemma} \label{lem:itscompat} The preorder $\le$ on $X_w$ is compatible with concatenation in that, for any words $A$ and $B$, $\un{w} \le \un{w}'$ if and only if $A \un{w} B \le
A \un{w}' B$. It is compatible with inclusions in that, for any $n, m, k$ such that $\iota = \iota_k^{n \to m}$ is defined, $\un{w} \le \un{w}'$ implies $\iota(\un{w}) \le
\iota(\un{w}')$. Finally, $\le$ satisfies the DCC. \end{lemma}

\begin{proof} We check compatibility with concatenation, as the remaining properties are a simple exercise. We claim that this follows from the fact that the partial order $\le_3$ on words in $S_3$ is compatible with
concatenation. For any $i < j < k$, tracing three strands through the composition $A \un{w} B$ is the composition of tracing three strands through $B$, then through $\un{w}$, then through $A$, although the labels on these strands change. More precisely, if $\un{w}$ expresses $w \in S_m$ then
\begin{equation} \label{AwBijk} (A\un{w} B)_{ijk} = A_{i''j''k''} \un{w}_{i'j'k'} B_{ijk}, \end{equation}
where $B(\{i,j,k\})= \{i',j',k'\}$ and $w(\{i', j', k'\}) = \{i'', j'', k''\}$. If $\un{w} \le \un{w}'$ and both have the same length, then they express the same element $w$, so that $(A \un{w} B)_{ijk}$ and $(A \un{w}' B)_{ijk}$ both obey \eqref{AwBijk} with the same primed indices.
Then $(A\un{w} B)_{ijk} \le (A\un{w}' B)_{ijk}$ follows because $\un{w}_{i'j'k'} \le \un{w}'_{i'j'k'}$. \end{proof}

\begin{lemma} \label{lem:equivmeans} For $\un{w}, \un{w}' \in X_w$, we have $\un{w} \le \un{w}' \le \un{w}$ if and only if $\un{w}_{ijk} = \un{w}'_{ijk}$ for all triples $(i|j|k)$. If so, $\un{w}$ and $\un{w}'$ have the same length, and $n_{ij} = n'_{ij}$ for all pairs $(i|j)$. \end{lemma}

\begin{proof} Most of this lemma is obvious, since $\le_3$ is a partial order so $A \le_3 B \le_3 A$ implies $A=B$. The only thing to show is that if $\un{w}_{ijk} = \un{w}'_{ijk}$ for
all $(i|j|k)$ then $\un{w}$ and $\un{w}'$ have the same length. But if $\un{w}_{ijk} = \un{w}'_{ijk}$ then $n_{ij} = n'_{ij}$. Summing over all pairs $(i|j)$, we see that they have the
same overall length. \end{proof}

We also prove, but do not use, the following observation.

\begin{lemma} If $\un{w}_{ijk} \le_3 \un{w}'_{ijk}$ for all triples $(i|j|k)$, then the length of $\un{w}$ is at most the length of $\un{w}'$. \end{lemma}

\begin{proof} Suppose $w$ is a permutation in $S_m$. Then the statement is obvious for $m \le 3$.
	
Because $\un{w}_{ijk} \le_3 \un{w}'_{ijk}$, either $\un{w}_{ijk}$ has shorter length and
\begin{equation} \label{eq:foobar1} n_{ij} + n_{ik} + n_{jk} < n'_{ij} + n'_{ik} + n'_{jk},\end{equation}
or $\un{w}_{ijk}$ and $\un{w}'_{ijk}$ have the same length, and
\begin{equation} \label{eq:foobar2} n_{ij} = n'_{ij}, \quad n_{ik} = n'_{ik}, \quad n_{jk} = n'_{jk}. \end{equation}
If \eqref{eq:foobar2} holds for all triples $(i|j|k)$ then clearly the lengths of $\un{w}$ and $\un{w}'$ are equal. So suppose \eqref{eq:foobar1} holds for some triple. Taking the sum over all triples, we obtain
\begin{equation} \sum_{(i|j|k)} n_{ij} + n_{ik} + n_{jk} < \sum_{(i|j|k)} n'_{ij} + n'_{ik} + n'_{jk}. \end{equation}
However, the left hand side is precisely $(m-2) \sum_{(i|j)} n_{ij}$, which is $(m-2)$ times the length of $\un{w}$. Thus the length of $\un{w}$ is smaller than the length of $\un{w}'$. \end{proof}

\begin{prop} One has $\un{w} \le \un{w}' \le \un{w}$ if and only if these two expressions are equivalent in $\bar{X}_w$. Therefore, $\le$ descends to a partial order on $\bar{X}_w$. \end{prop}

\begin{proof} One direction is easy. The commuting braid relations $s_i s_j = s_j s_i$ are not detected by any triples, and do not change the length of the expression. Thus two expressions equivalent in $\bar{X}_w$ are equivalent under $\le$. We now prove the other direction. Let $\un{w} \le \un{w}' \le \un{w}$. By Lemma \ref{lem:equivmeans}, we know that $n_{ij} = n'_{ij}$ for all $(i|j)$, and $\un{w}_{ijk} = \un{w}'_{ijk}$ for all $(i|j|k)$.

Suppose that $s_i$ is the first simple reflection in $\un{w}'$ (i.e. it appears on the far right of the word, reading right to left, or on the bottom of a crossing diagram). Clearly
$s_i$ appears somewhere in $\un{w}$, or the crossing $(i|i+1)$ could never happen. We claim that an instance of $s_i$ appears before any $s_{i+1}$ or $s_{i-1}$ in $\un{w}$. Suppose that
$s_{i+1}$ appears before the first $s_i$ in $\un{w}$. Before the first instance of $s_{i+1}$, neither $s_i$ nor $s_{i+1}$ appears, so the strands $\{i+2, i+3, \ldots, m\}$ are permuted
amongst themselves (where $w \in S_m$), and strand $i+1$ is fixed. Therefore, the first $s_{i+1}$ has crossing type $(i+1,j)$ for some $j \in \{i+2, \ldots, m\}$. But then a crossing
$(i+1,j)$ appears before $(i|i+1)$. This contradicts the fact that $\un{w}_{(i|i+1|j)} = \un{w}'_{(i|i+1|j)}$. A similar contradiction arises if $s_{i-1}$ appears before $s_i$.

Thus, a sequence of commuting braid relations $s_i s_j = s_j s_i$ will send $s_i$ to the start of $\un{w}$. Up to equivalence, we can assume $\un{w}$ also started with $s_i$. Removing this index $s_i$ from the start of both expressions, the inequality $\le$ still holds by Lemma \ref{lem:itscompat}. We can then use induction on the length of expressions to prove that
$\un{w}$ and $\un{w}'$ are equivalent in $\bar{X}_w$.
\end{proof}

\begin{lemma} The partial order $\le$ on $\bar{X}_w$ respects the orientation of Definition \ref{def:orientXw}. \end{lemma}
	
\begin{proof} Applying the relation $ss \mapsto 1$ will shorten the length of a word. Applying the relation $s_\ell s_{\ell+1} s_\ell \mapsto s_{\ell+1} s_\ell s_{\ell+1}$ will affect
precisely one triple $(i|j|k)$, and will send it to something smaller in $\le_3$. \end{proof}

\begin{thm} \label{thm:sink} The oriented graph $\bar{X}_w$ has a unique sink, which is a reduced expression. \end{thm}
	
\begin{proof} Suppose that the equivalence class $[\un{w}]$ of $\un{w}$ is a sink. If $\un{w}$ is a reduced expression, then $[\un{w}]$ must be the unique sink of $\bar{\G}_w$ by Proposition \ref{prop:MS}. Thus we suppose that $\un{w}$ is not reduced, and derive a contradiction using the following lemma. \end{proof}
	
\begin{lemma} Suppose that $\un{w} \in X_w$ is not reduced. Then $[\un{w}]$ is not a sink.
\end{lemma}
	
\begin{proof} We show this by induction on the length of $\un{w}$. The base case is trivial. If any proper subword of $\un{w}$ is not a reduced expression, then by induction it is not a
sink. Hence (using the compatibility of the orientation with concatenation), $\un{w}$ is also not a sink. Thus we can assume each subword of $\un{w}$ is a reduced expression. Moreover,
using Proposition \ref{prop:MS}, we can assume that each subword of $\un{w}$ is (a representative of the equivalence class of) the unique sink in its reduced expression graph.

So suppose $\un{w} = s_i \un{y} s_j$, for some reduced expression $\un{y}$ for $y \in S_n$. Both $s_i \un{y}$ and $\un{y} s_j$ are reduced. By the exchange condition for the symmetric
group (see \cite[\S 1.7]{HumpCox}), we see that $s_i$ and $s_j$ must cross the same pair of strands $(j|j+1)$. In particular, $y(j) = i$ and $y(j+1) = i+1$. Therefore, $y s_j = s_i y$ and
$y = w$.
\[ {
\labellist
\small\hair 2pt
 \pinlabel {$s_j$} [ ] at 1 18
 \pinlabel {$\un{y}$} [ ] at 1 45
 \pinlabel {$s_i$} [ ] at 1 72
\endlabellist
\centering
\ig{1.3}{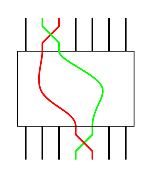}
} \]
(The reader who prefers a non-trivial base case will content themself with the case where $\un{y}$ is the empty word, and $i=j$.)

Suppose that $k > j+1$. We claim that $y(k) > i+1$. Otherwise, $y(k) < i$ and the triple $(j|j+1|k)$ is full in $ys_j$. However, the reduced expression $\un{y} s_j$ is the sink in its reduced expression graph, so by Corollary \ref{cor:MS} the packet of $(j|j+1|k)$ must appear in antilexicographic order. However, $(j|j+1)$ appears first, a contradiction.

Suppose that $k < j$. We claim that $y(k) < i$. Otherwise, $y(k) > i+1$ and the triple $(k|j|j+1)$ is full in $s_i y$. However, the reduced expression $s_i \un{y}$ is the sink in its reduced expression graph, so by Corollary \ref{cor:MS} the packet of $(k|j|j+1)$ must appear in antilexicographic order. However, $(j|j+1)$ appears last, a contradiction.

By counting, we see that the number of strands to the right of $j+1$ (resp. the left of $j$) must equal the number of strands to the right of $i+1$ (resp. the left of $i$). Thus $i=j$.
Then $y$ is some permutation in $S_{i-1} \times S_1 \times S_1 \times S_{m-i-1} \subset S_m$. Every simple reflection $s_k$ in the reduced expression $\un{y}$ satisfies $k \ne i \pm 1$,
and thus commuting braid relations will bring $\un{w}$ to the form $\un{y} s_i s_i$.

Finally, the oriented edge $s_i s_i \mapsto 1$ indicates that $\un{w}$ is not a sink.
\end{proof}

\begin{remark} The proof of this lemma actually gives an algorithm for taking a non-reduced expression and shortening it, using only the moves $s_i s_{i+1} s_i \mapsto s_{i+1} s_i
s_{i+1}$, $s_i s_j = s_j s_i$, and $s_i s_i \mapsto 1$. Namely, look at some non-reduced subword $\un{w} = s_i \un{y} s_j$ where $s_i \un{y}$ and $\un{y} s_j$ are reduced. Shorten
$\un{y}$ as much as possible by using commuting braid relations to bring simple reflections to the left past $s_i$ or the right past $s_j$. If $\un{y}$ is empty, then $i=j$ and $s_i
s_i\mapsto 1$ suffices. Otherwise, either some $k>j+1$ has $y(k) < i$, in which some packet of $\un{y} s_j$ is flippable, or some $k < j$ has $y(k) > i+1$, in which case some packet of
$s_i \un{y}$ is flippable. Repeat this operation for the win. \end{remark}

\section{The diamond lemma for monoidal categories}
\label{sec:monoidalcats}

A natural setting wherein the set of expressions modulo equivalence $\bar{X}_w$ is more natural than the set of expressions $X_w$ is the setting of monoidal categories. This shall be
explained further below. Thus, we seek a diamond lemma for monoidal categories.

It would be very surprising if this had not been done before, but we could not find it. See \cite{DotKho} for another extension of Gr\"obner bases to a more complicated setup.

{\bf Addendum}. Diamond lemmas for monoidal categories have appeared in the literature, see the references in \cite{Dupont}.

\subsection{Monoidal categories}
\label{subsec:monoidal}

We turn our attention to $\Bbbk$-linear (strict) monoidal categories. These are equipped with an additional structure, \emph{horizontal composition} $\ot$, which is also
$\Bbbk$-bilinear like ordinary composition (which we call \emph{vertical composition}). String diagrams are an effective tool for describing elements of a strict monoidal category. See
\cite{LauDiagrams,Marsden} for more background.

The compatibility laws between horizontal and vertical multiplication in a monoidal category state that, when $f \co n \to n'$ and $g \co m \to m'$, one has $(f \ot \id_{m'}) \circ
(\id_{n} \ot g) = (f \ot g) = (\id_{n'} \ot g) \circ (f \ot \id_m)$ as maps $n \ot m \to n' \ot m'$. The effect of this law on string diagrams is to state that two string diagrams which
are equivalent under \emph{rectilinear isotopy} represent morphisms which are equal.

\[{
\labellist
\tiny\hair 2pt
 \pinlabel {$f$} [ ] at -3 22
 \pinlabel {$g$} [ ] at 8 9
 \pinlabel {$f$} [ ] at 32 16
 \pinlabel {$g$} [ ] at 45 16
 \pinlabel {$f$} [ ] at 68 9
 \pinlabel {$g$} [ ] at 80 22
 \pinlabel {$=$} [ ] at 22 14
 \pinlabel {$=$} [ ] at 59 14
\endlabellist
\centering
\ig{1}{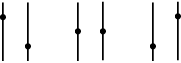}
}\]

One can present a monoidal algebra by generators and relations. One thinks of the generating morphisms as elementary string diagrams, and relations as local rules which can be applied to
any subdiagram of a diagram.

\begin{example} \label{ex:Scat2} The category $\SymCat$ from Example \ref{ex:Scat} is monoidal. Horizontal composition is induced by the map $S_n \times S_m \to S_{n + m}$, which gives a map $\Bbbk[S_n] \ot \Bbbk[S_m] \to \Bbbk[S_{n+m}]$.

Its presentation as a monoidal category is much simpler than its presentation as a $\Bbbk$-linear category. There is one generating object $1 \in \NM$, and one generating morphism $\s
\in \End(2)$, drawn as a crossing. There are two relations, corresponding to \eqref{eq:s2} and \eqref{eq:sisjsi}: $\s \s = \id_2$, and $(\s \ot \id_1) \circ (\id_1 \ot \s) \circ (\s \ot
\id_1) = (\id_1 \ot \s) \circ (\s \ot \id_1) \circ (\id_1 \ot \s)$ in $\End(3)$. These relations are drawn as follows. \begin{equation} \label{eq:symcatrel} {
\labellist
\small\hair 2pt
 \pinlabel {$=$} [ ] at 15 21
 \pinlabel {$=$} [ ] at 87 21
\endlabellist
\centering
\ig{1}{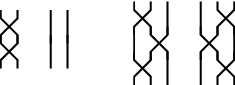}
} \end{equation} By convention, when discussing a morphism in $\End(2)$ we write $\s$, but when discussing a morphism in $\End(3)$ we write $s = \s \ot \id_1$ and $t = \id_1 \ot \s$, so that the second relation is $sts = tst$. We also write $s, t, u$ for the diagrams corresponding to the three simple reflections in $S_4$, viewed as elements of $\End(4)$. For example, the equality $su = us$ is a consequence of rectilinear isotopy, so is not needed as a relation. \end{example}

A presentation of a monoidal category does give rise to a presentation of its endomorphism algebras as ordinary categories. For instance, one can recover the Coxeter presentation of
$S_n$ from the presentation above. The point of this discussion is that, for many monoidal categories, the monoidal presentation is adapted to a variant of the diamond lemma,
while the ordinary presentation is not adapted to the original diamond lemma.

\subsection{The diamond lemma for monoidal presentations}
\label{subsec:diamondmonoidal}

We seek a version of Bergman's diamond lemma which functions for monoidal presentations of $\Bbbk$-linear monoidal categories. Thankfully, Bergman's original proof functions almost
verbatim once one changes the notation, replacing ``linear'' descriptions of concatenation with ``planar'' ones. One replaces monomials with isotopy classes of diagrams, and replaces
concatenation (the composition of words) with plugging diagrams into each other (the composition of diagrams).

Let $\CC$ be a strict monoidal category with generating objects $\NC$, generating morphisms $\SC$ and local relations $\RC$. In particular, the objects of $\CC$ are words in $\NC$, but
the objects will not play a major role in this discussion. Let $X$ denote the set of all diagrams built from the generating morphisms $\SC$, and $X(n,m)$ the diagrams with source $n$ and
target $m$, for words $n, m$ in the alphabet $\NC$. Let $\bar{X}$ denote the quotient of $X$ by the equivalence relation given by rectilinear isotopy.

Let us fix a partial order $\le$ on $\bar{X}$. We assume that $\bar{X}(n,m)$ and $\bar{X}(n',m')$ are incomparable unless $n=n'$ and $m = m'$, i.e. the partial order \emph{respects
morphism spaces}. It is a \emph{monoidal partial ordering} if the following condition holds: if $\phi$ is obtained by plugging $B$ into a larger diagram, and $\psi$ is obtained by
plugging $B'$ into the same larger diagram, and $B < B'$, then $\phi < \psi$. This is equivalent to saying that $\le$ respects vertical concatenation on either top or bottom (c.f. the
previous notion of a semigroup partial order), and respects horizontal concatenation on either left or right with the identity map of any object. We say that $\le$ is \emph{compatible
with $\RC$} if every relation $r \in \RC$ can be rewritten as an equation $W_r = f_r$, where $W_r \in \bar{X}$, and $f_r$ is a $\Bbbk$-linear combination of various $A \in \bar{X}$
satisfying $A < W_r$. In the statement of the diamond lemma, we also need that $\le$ satisfies the DCC.

Given such a partial order, we can define \emph{elementary resolutions}, \emph{resolutions}, and \emph{irreducible} elements of $\bar{X}$, just as in \S\ref{subsec:originaldiamond}. We
can also define \emph{ambiguities} $(A,W_r,W_s)$, which appear when a diagram $A$ has two distinct subdiagrams $W_r$ and $W_s$ for $r, s \in \RC$, or the same subdiagram when $W_r =
W_s$ and $r \ne s$. This leads to two possible elementary resolutions of $A$, and the ambiguity is \emph{resolvable} if they can be jointly resolved to the same linear combination of
isotopy classes of diagrams. One possibility is that the subdiagrams $W_r$ and $W_s$ are \emph{disjoint}, meaning that they do not intersect in the plane. As before, disjoint
ambiguities can be easily jointly resolved.

There are a vast number of ways in which non-disjoint ambiguities can overlap! After all, there are many interesting ways in which two simply-connected domains in the plane can
intersect (as in the picture below), and there is some crazy monoidal category with an ambiguity of each kind. \igc{1}{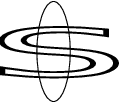} It is not worth classifying how diagrams can meet (as we
did in the linear case, e.g. overlap vs. inclusion), as every overlap will be treated in an identical way, regardless of how complicated the picture is.

However, there is a major complication that comes with planar diagrams rather than linear words. In the original setting there are certain \emph{minimal} examples of each ambiguity,
which give rise to all possible ambiguities. For example, when $W_r = BC$ and $W_s = CD$ overlap, then the minimal ambiguity $BCD$ appears inside every other ambiguity $LBCDM$.
Consequently, if the minimal ambiguity is resolvable, then so is any overlap ambiguity of this type. For general monoidal categories, it is far less easy to find a collection of minimal
ambiguities which account for every possible ambiguity! In the picture above, any diagram with the appropriate boundary could go in the ``gaps'' between the subdiagrams $W_r$ and $W_s$. Let us give a relevant example which illustrates the difficulty.

\begin{example} \label{ex:ststsBad} Consider the presentation of $\SymCat$ in Example \ref{ex:Scat2}, and suppose we plan to apply the relations in the direction $\s\s \mapsto 1$ and $sts
\mapsto tst$. There is the potential that $st{\bf s}$ overlaps with ${\bf s}ts$ in the bold copies of $s$, but the ways in which this overlap can occur can be complicated, because the ``loose ends'' in the relations can interact with other diagrams. \begin{equation} \label{eq:stsstsoverlap1} {
\labellist
\small\hair 2pt
 \pinlabel {$A = $} [ ] at -15 50
 \pinlabel {$B$} [ ] at 51 50
\endlabellist
\centering
\ig{1}{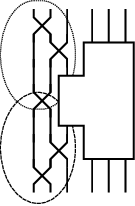}
} \end{equation} In
this picture, $B$ can be any diagram with the appropriate boundary. There is no single minimal diagram which contains both copies of $sts$, and is contained in any such ambiguity $A$!
Nonetheless, below we will find two ``minimal ambiguities'' associated with this overlap which, if resolvable, imply that any such overlap is resolvable. \end{example}

\begin{example} \label{ex:sstsNotBad} Sometimes life is not so bad. Any ambiguity of the form $(A, s{\bf s}, {\bf s}ts)$, where the bold copies of $s$ overlap, will contain the
following subdiagram, because all loose ends are accounted for. \igc{1}{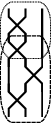} This ambiguity $(ssts, s{\bf s}, {\bf s}ts)$ is therefore a minimal such ambiguity. \end{example}

We say that a family $\{(A_i, W_r, W_s)\}$ is a \emph{minimal set of ambiguities} if, when they are all resolvable, every ambiguity is resolvable. The word ``minimal'' should be taken with a grain of salt: for example, one could choose the set of all ambiguities. Typically, though, one can do better.

\begin{thm}[Diamond Lemma for monoidal categories] \label{thm:MDL} Let $\CC$ be a $\Bbbk$-linear monoidal category for a commutative base ring $\Bbbk$, presented by generators and
relations. Suppose that one has a partial order $\le$, respecting the morphism spaces, which is a monoidal ordering, is compatible with $\RC$, and satisfies the DCC. Then the irreducible
diagrams $\bar{X}_{\irr}$ form a basis for $\CC$ if and only if each ambiguity is resolvable, if and only if each ambiguity is resolvable relative to $\le$. If there is a minimal set of ambiguities for each type of ambiguity, then these conditions are equivalent to all the minimal ambiguities being resolvable, or resolvable relative to $\le$. If these conditions are true, we say that the data of the presentation and the partial order is \emph{monoidal Bergman type}. \end{thm}

\begin{proof} The proof of \cite[Theorem 1.2]{Bergman} translates almost verbatim to this setting, replacing $X$ with $\bar{X}$, words with diagrams, etcetera. \end{proof}

The main difference between Theorems \ref{thm:BDL} and \ref{thm:MDL} is the lack of any automatic and reasonably small minimal set of ambiguities. This makes the adaptation
to any given setting somewhat trickier than the original Bergman Diamond Lemma, which can be applied almost without thinking (although with the caveat that one must first find a partial
order $\le$). The point of the diamond lemma is that it enables one to check if $\bar{X}_{\irr}$ is a basis with only a finite amount of work, and until one finds finitely many minimal
ambiguities, the advantage of the diamond lemma is not really present. So there is one additional piece of work which needs to be done for a monoidal presentation: find a reasonably
small minimal set of ambiguities for each possible overlap.

\section{Applications to Hecke-type categories}
\label{sec:hecketype}

Let us give a concrete application of Theorem \ref{thm:MDL}, where a small minimal set of ambiguities is found. 

\subsection{Applications to the symmetric group}
\label{subsec:sym}

\begin{lemma} \label{lem:Sambenough} For the presentation of $\SymCat$ in Example \ref{ex:Scat2}, the list below gives a minimal set of ambiguities. \begin{equation} \label{eq:Samb} \ig{1.3}{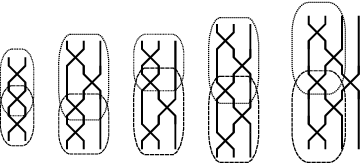} \end{equation} \end{lemma}

\begin{proof} Consider an ambiguity $(A,B,C)$. Obviously, if a proper subdiagram $A'$ of $A$ contains both $B$ and $C$, then using the fact that $A' < A$ (because it has fewer crossings) we see that the difference between the two elementary resolutions lies within $I_{< A}$. As discussed in Example \ref{ex:sstsNotBad}, whenever the ambiguity has $B = s{\bf s}$ and $C = {\bf s}ts$, the corresponding ambiguity $ssts$ in \eqref{eq:Samb} is a proper subdiagram of $A$. The same is true for the ambiguities $\s \s \s$ and $stss$. All that remains is to discuss the ambiguities of the form in \eqref{eq:stsstsoverlap1}.
	
Let us rewrite such an ambiguity as follows.
\begin{equation} \label{eq:stsstsoverlap2} {
\labellist
\small\hair 2pt
 \pinlabel {$A = $} [ ] at -15 50
 \pinlabel {$Z$} [ ] at 51 50
 \pinlabel {$X$} [ ] at 51 79
 \pinlabel {$Y$} [ ] at 51 19
\endlabellist
\centering
\ig{1}{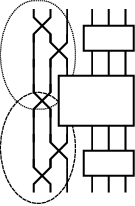}
} \end{equation}
Up to taking subdiagrams, we can assume that what appears in regions $X$ and $Y$ is trivial, and what appears in region $Z$ is an expression for a permutation in $S_k$ for some $k$.
Within the ideal $I_{<A}$ we can apply any known reductions to $Z$. More precisely, each diagram in the two elementary resolutions of $A$ will be strictly less than $A$, so that applying relations to the copy of $Z$ in each such diagram is adding an element of $I_{<A}$. Thus, we can assume inductively that $Z$ is a reduced expression, and even a reduced expression of our choice!

Any permutation in $S_k$ has a reduced expression where the first simple reflection $s_1$ appears either once or not at all. That is, $S_k = S_{k-1}
\cup S_{k-1} s_1 S_{k-1}$, where $S_{k-1}$ refers to the permutations of all but the first strand. Any permutation in $S_{k-1}$ may be assumed to be in region $X$ or $Y$; that is, up
to taking subdiagrams, we can assume that $Z$ is either the identity or is $s_1$. This leads to the two remaining minimal ambiguities in \eqref{eq:Samb}, $ststs$ and $stsuts$. \end{proof}

\begin{remark} To be very clear, not every diagram with two overlapping copies of $sts$ will have $stsuts$ as a subdiagram. After all, $Z$ can be any expression in $S_k$. The point is
that, if the ambiguities of \eqref{eq:Samb} are resolvable, then so is any other ambiguity. \end{remark}

Of course, the ambiguities in \eqref{eq:Samb} are resolvable. Because it will be useful to keep in mind, we demonstrate the resolvability of $stsuts$.
\begin{subequations}
\begin{equation} \label{eq:path1} {
\labellist
\small\hair 2pt
 \pinlabel {$\mapsto$} [ ] at 39 39
 \pinlabel {$\mapsto$} [ ] at 87 39
 \pinlabel {$\mapsto$} [ ] at 135 39
 \pinlabel {$\mapsto$} [ ] at 183 39
\endlabellist
\centering
\ig{1.3}{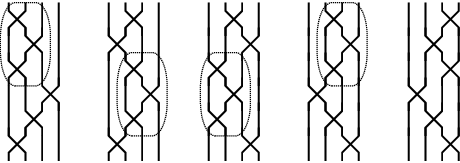}
} \end{equation}
\begin{equation} \label{eq:path2} {
\labellist
\small\hair 2pt
 \pinlabel {$\mapsto$} [ ] at 39 39
 \pinlabel {$\mapsto$} [ ] at 87 39
 \pinlabel {$\mapsto$} [ ] at 135 39
 \pinlabel {$\mapsto$} [ ] at 183 39
\endlabellist
\centering
\ig{1.3}{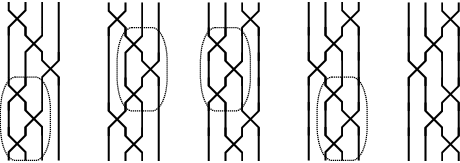}
} \end{equation}
\end{subequations}
Thus, one follows the graph $\bar{\G}_{stsuts}$ from source to sink along its two different paths, and confirms that the answer is the same.

However, to apply Theorem \ref{thm:MDL}, we still need the existence of a suitable partial order $\le$. We have already constructed this order in Definition \ref{def:partial}, and proven that $\bar{X}_{\irr} = \coprod \bar{X}_{w, \irr}$ has precisely one sink for each $w \in S_n$. Thus we have the following very unsurprising theorem.

\begin{thm} The Coxeter presentation is actually a presentation of $\SymCat$ as claimed. That is, $\End(n)$ has a basis $\{\un{w}\}$ indexed by $w \in S_n$, obtained by taking an arbitrary reduced expression $\un{w}$ for each $w$. \end{thm}

\begin{proof} Theorem \ref{thm:MDL} implies that one has a basis $\{\un{w}\}$, where $\un{w}$ is a very particular reduced expression for each $w$, namely the sink of $\bar{X}_w$. However, from here it is easy to see that any arbitrary reduced expression will do. \end{proof}

\subsection{Applications to modified symmetric categories}
\label{subsec:symvar}

This same exact proof generalizes to a broader class of presentations.

\begin{defn} Consider a $\Bbbk$-linear monoidal category $\CC$ with the following type of presentation, which we call a \emph{modified symmetric category}. It has one generating object $1$, and one generating morphism $\s \in \End(2)$, drawn as a crossing. The relations are:
\begin{subequations} \label{subeq:coxmod}
\begin{equation} \label{eq:s2mod} \s \s = \a \s + \b \id_2 \quad \textrm{ for some } \a, \b \in \Bbbk. \end{equation}
\begin{equation} \label{eq:R3mod} sts = tst + a st + b ts + c s + d t + e \id_3 \quad \textrm{ for some } a, b, c, d, e \in \Bbbk. \end{equation}
Here, $s = \s \ot \id_1$ and $t = \id_1 \ot \s$, both in $\End(3)$. Using diagrams, these are:
\begin{equation} \label{eq:s2modDiag}{
\labellist
\small\hair 2pt
 \pinlabel {$= \a$} [ ] at 20 14
 \pinlabel {$+ \b$} [ ] at 52 14
\endlabellist
\centering
\ig{1}{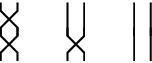}
}, \end{equation}
\begin{equation} \label{eq:R3modDiag}{
\labellist
\small\hair 2pt
 \pinlabel {$=$} [ ] at 27 20
 \pinlabel {$+ a$} [ ] at 67 20
 \pinlabel {$+ b$} [ ] at 107 20
 \pinlabel {$+ c$} [ ] at 147 20
 \pinlabel {$+ d$} [ ] at 187 20
 \pinlabel {$+ e$} [ ] at 227 20
\endlabellist
\centering
\ig{1}{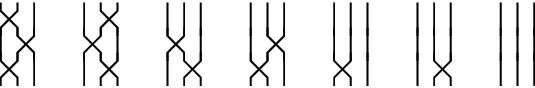}
}. \end{equation}
\end{subequations}
\end{defn}

\begin{prop} \label{prop:MDLforScat} Let $\CC$ be a modified symmetric category. Then $\CC$ has a basis indexed by $w \in S_n$, obtained by taking an arbitrary reduced expression for each $w$, if and only if each of the ambiguities in \eqref{eq:Samb} is resolvable. \end{prop}
	
\begin{proof} The only thing which needs proving is the existence of a suitable partial order, with the correct list of irreducible diagrams. The partial order $\le$ of Definition
\ref{def:partial}, modified as in Remark \ref{rmk:partialextended}, will clearly suffice, as each relation is identical to the Coxeter relations modulo words of shorter length.
\end{proof}

Let us examine what checking the ambiguities entails.

The ambiguity $\s \s \s$ is always resolvable, as both elementary resolutions yield $\a \s \s + \b \s$.
	
Checking the ambiguity $ssts$ directly is a chore. One can resolve it into linear combinations of $\{1, s, t, st, ts, tst\}$ in two ways. Let us examine the coefficient of $tst$ which occurs (ignoring any terms of shorter length).
\begin{subequations} \label{subeq:sstsbasictop}
\begin{equation} \un{ss}ts \mapsto \a sts \mapsto \a tst. \end{equation}
\begin{equation} s\un{sts} \mapsto stst + b sts \mapsto tstt + b tst + b tst \mapsto (\a + 2b) tst. \end{equation}
\end{subequations}
Thus, for \eqref{subeq:sstsbasictop} to agree, one requires $2b=0$.  Let us assume that $\Bbbk$ has no 2-torsion for the rest of the calculation, so that $b = 0$.

Checking the $ts$ coefficient in the same way yields the equality $\b + \a b = \b + b^2$, which holds when $b = 0$.

Similarly, checking the $tst$ coefficient in the ambiguity $stss$ yields $2a=0$. In general, checking the $w$ coefficient of $stss$ is the same as checking the $w^{-1}$ coefficient of
$ssts$, after swapping $a$ and $b$, and swapping $c$ and $d$ (this is clear from the word-reversing symmetry). So we set $a = 0$ for the rest of the calculation.

Let us now check the $st$ coefficient of $ssts$, assuming that $a = b = 0$.
\begin{subequations} \label{subeq:sstsbasicnext}
\begin{equation} \un{ss}ts \mapsto \a sts \mapsto 0. \end{equation}
\begin{equation} s\un{sts} \mapsto stst + d st \mapsto tstt + c st + d st \mapsto (c+d) st. \end{equation}
\end{subequations}
Thus the agreement of \eqref{subeq:sstsbasicnext} implies $c+d = 0$.

We leave the reader to verify that the $s$ coefficient of $ssts$ yields the equality $e = 0$, as does the $t$ coefficient. The $1$ coefficient yields the equality $\b(c+d) = 0$, which
is nothing new.

In conclusion, the ambiguities $ssts$ and $stss$ are resolvable if and only if $a = b = e = c+d = 0$, which we assume henceforth. We leave the reader to confirm that the ambiguity $ststs$ is resolvable under these assumptions.

To check if the ambiguity $stsuts$ is resolvable, one follows the two different paths from source to sink in $\bar{\G}_{stsuts}$. This time, lower terms are produced by
\eqref{eq:R3mod}, and one needs to check if the lower terms agree. We leave the reader to verify that the lower terms in \eqref{eq:path1} are
\begin{subequations}
\begin{equation} c suts + d tuts + c tsts + d tsus + c tusu + d tutu + c tstu + d ustu, \end{equation}
while the lower terms in \eqref{eq:path2} are
\begin{equation} c stus + d stut + c stst + d sust + c usut + d utut + c utst + d utsu. \end{equation}
Resolving each of these expressions further, taking the difference, and using $d = -c$, one eventually obtains
\begin{equation} \label{eq:stsutsbasic} c\b (st - ts + tu - ut). \end{equation}
\end{subequations}
Thus $stsuts$ is resolvable (given the assumptions above) if and only if $c \b = 0$.

\begin{thm} Assume that $\Bbbk$ has no 2-torsion. Then a modified symmetric category $\CC$ with the partial order $\le$ is monoidal Bergman type if and only if $a=b=e=c+d =0$ and $c \b
= 0$. \end{thm}

\begin{proof} This follows from Proposition \ref{prop:MDLforScat}. The computations above checked all the ambiguities of \eqref{eq:Samb}. \end{proof}
	
In particular, there are very few isomorphism classes of Bergman type modified symmetric categories.

\subsection{Applications to Hecke-type categories}
\label{subsec:hecke}

Once more, the same proof as in \S\ref{subsec:sym}, combined with the observations of \S\ref{subsec:noncentral}, generalizes easily to a much broader class of categories. Checking the
ambiguities is more of a chore, and will be done (under some minor additional assumptions) in \cite{EGarcia1}, but we should emphasize that there are many more possibilities than in
\S\ref{subsec:symvar}.

The following definition encapsulates many categories in the literature, like Khovanov-Lauda-Rouquier algebras \cite{KhoLau09, Rouq2KM-pp} and Webster algebras \cite{WebsKIHRT}.

\begin{defn} \label{def:HeckePart1} A \emph{Hecke-type category} is a strict monoidal category $\CC$ with the following type of presentation. The generating objects are a set $I$, so
that objects are words $\un{i}$ in the alphabet $I$. We think of the elements of $I$ as \emph{colors}. The morphisms are generated by: \begin{itemize} \item For each $i \in I$, some
commutative ring $R_i$ of maps in $\End(i)$. We think of these maps as ``dots,'' analogous to the dots in the Khovanov-Lauda calculus. Here is the picture for multiplication by $f \in R_i$, when $i$ is the color red.
\begin{equation} \label{eq:dotpic} {
\labellist
\small\hair 2pt
 \pinlabel {$f$} [ ] at 4 14
 \pinlabel {$\textrm{or}$} [ ] at 26 14
 \pinlabel {$f$} [ ] at 46 14
\endlabellist
\centering
\ig{1}{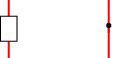}
} \end{equation} \item For some subset $C$ of ordered
pairs $(i,j) \in I \times I$, there is a \emph{crossing} $\s_{ij} \co ij \to ji$. Here is the picture, when $i$ is red and $j$ is blue. \begin{equation} \label{eq:redbluecross} \ig{1.3}{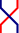} \end{equation} \end{itemize} If $(i,j)$ is not in the subset $C$, then $\s_{ij}$ is not a generator, and we say that crossing $i$ over $j$ is \emph{not permitted}.

With these generators, a diagram is a crossing diagram, where the strands are each colored by an element of $I$, and only certain colors are permitted to cross. In addition, the diagram
is decorated with dots in various places; here is an example.
\[{
\labellist
\tiny\hair 2pt
 \pinlabel {$f$} [ ] at -4 20
 \pinlabel {$g$} [ ] at 13 5
 \pinlabel {$h$} [ ] at 13 31
\endlabellist
\centering
\ig{1.3}{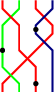}
}\] We refer to a diagram with the dots removed as the \emph{underlying crossing diagram} or \emph{underlying expression}.

The relations will be discussed below. \end{defn}

\begin{remark} Note that we have not decided to give a presentation for the rings $R_i$. Instead, we prefer to think of the rings $R_i$ as a black box,
along the lines of the discussion in \S\ref{subsec:noncentral}. \end{remark}

Let $\un{i}$ and $\un{j}$ be two sequences of colors of length $n$. We call $w \in S_n$ a \emph{permissible permutation} if $w(\un{i}) = \un{j}$, and the crossings in a reduced
expression of $w$ are all permissible. This does not depend on the reduced expression, only on the inversion set $I(w)$; the colors which cross in an expression form a subset of $I
\times I$ which is not changed by any braid relation. More generally, we refer to a (non-reduced) expression of a permutation as \emph{permissible} if the crossings are all permissible.
It is clear that every diagram representing a morphism from $\un{i}$ to $\un{j}$ has underlying crossing diagram given by a permissible expression, and every permissible expression
gives such a diagram.

\begin{example} Consider diagrams $ij \to ij$ with $i \ne j$. Only the identity is a permissible permutation. If either the crossing $\s_{ij}$ or the crossing $\s_{ji}$ is not
permitted, then only the empty word is a permissible expression. If both are permitted, then the permissible expressions are $\s^k$ for $k$ even, where $\s$ is the generator of $S_2$.
\end{example}

For any sequence $\un{i}$, the commutative ring $R_{\un{i}} = R_{i_1} \sqot R_{i_2} \sqot \cdots \sqot R_{i_d}$ is a subring of $\End(\un{i})$, given by decorations of the identity
crossing diagram. We tend to use the rectangle notation from \eqref{eq:dotpic}, rather than the dot notation, to more efficiently encode an element of $R_{\un{i}}$ (see the examples
below). We are interested in imposing relations upon these generators, such that the morphism spaces $\Hom(\un{i},\un{j})$ are free as $R_{\un{j}}$-modules on the left (and as
$R_{\un{i}}$-modules on the right), with a basis given a fixed reduced expression for each permissible permutation. For this to occur, we need certain relations: a quadratic relation
like \eqref{eq:s2mod}, a braid relation like \eqref{eq:R3mod} (that is, we need the Coxeter relations modulo shorter expressions), and a dot-moving relation like \eqref{eq:sf}. However,
in each of these relations, the coefficients that appear (like $a, b, c, d, e$ in \eqref{eq:R3mod}) can be elements of $R_{\un{j}}$. Moreover, only those shorter expressions which are
permissible can appear in such a relation.
	
\begin{defn*}[Definition \ref{def:HeckePart1} Continued] The relations in a Hecke-type category have the following form. Once the source and target of a morphism are understood, we may write $s$ and $t$ for the corresponding elements of $S_3$, viewed as crossing diagrams (independent of color). In the drawings below, $i$ is red, $j$ is blue, and $k$ is green.

\begin{subequations} \label{subeq:HeckeRelns}
If $\s_{ii}$ is permitted then one has the following relation in $\End(ii)$.
\begin{equation} \label{eq:iiHecke} \s_{ii} \s_{ii} = \a_{i} \s_{ii} + \b_{i} \id_{ii} \end{equation}
\[{
\labellist
\small\hair 2pt
 \pinlabel {$=$} [ ] at 14 13
 \pinlabel {$\a_i$} [ ] at 28 20
 \pinlabel {$+$} [ ] at 40 13
 \pinlabel {$\b_i$} [ ] at 52 20
\endlabellist
\centering
\ig{1.3}{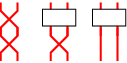}
}\]
for some $\a_{i}, \b_i \in R_{ii}$.

If $\s_{ij}$ and $\s_{ji}$ are permitted then one has the following relation in $\End(ji)$.
\begin{equation} \label{eq:ijHecke} \s_{ij} \s_{ji} = Q_{ji} \id_{ji} \end{equation}
\[{
\labellist
\small\hair 2pt
 \pinlabel {$=$} [ ] at 17 15
 \pinlabel {\tiny $Q_{ji}$} [ ] at 36 20
\endlabellist
\centering
\ig{1.3}{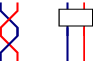}
}\]
for some $Q_{ji} \in R_{ji}$.

If $\s_{ii}$ is permitted then one has the following relation in $\End(iii)$.
\begin{equation} \label{eq:iiiHecke} sts = \l_{iii} tst + a_i st + b_i ts + c_i s + d_i t + e_i \end{equation}
\[{
\labellist
\tiny\hair 2pt
 \pinlabel {$= \l_{iii}$} [ ] at 28 20
 \pinlabel {$+$} [ ] at 68 20
 \pinlabel {$+$} [ ] at 108 20
 \pinlabel {$+$} [ ] at 148 20
 \pinlabel {$+$} [ ] at 188 20
 \pinlabel {$+$} [ ] at 228 20
 \pinlabel {$a_i$} [ ] at 88 32
 \pinlabel {$b_i$} [ ] at 128 32
 \pinlabel {$c_i$} [ ] at 168 32
 \pinlabel {$d_i$} [ ] at 208 32
 \pinlabel {$e_i$} [ ] at 248 32
\endlabellist
\centering
\ig{1.3}{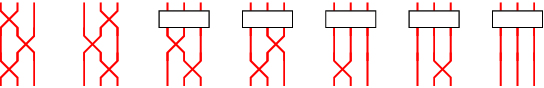}
}\]
for some $a_i,b_i, c_i, d_i, e_i \in R_{iii}$ and some invertible scalar $\l_{iii}$.

If $\s_{ii}$ and $\s_{ij}$ are permitted then one has the following relation in $\Hom(iij,jii)$.
\begin{equation} \label{eq:iijHecke} sts = \l_{jii} tst + p_{jii} st \end{equation}
\[{
\labellist
\small\hair 2pt
 \pinlabel {$= \l_{jii}$} [ ] at 28 20
 \pinlabel {$+$} [ ] at 68 20
 \pinlabel {\tiny $p_{jii}$} [ ] at 88 32
\endlabellist
\centering
\ig{1.3}{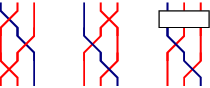}
}\]
for some $p_{jii} \in R_{jii}$ and some invertible scalar $\l_{jii}$.

If $\s_{ii}$ and $\s_{ji}$ are permitted then one has the following relation in $\Hom(jii, iij)$.
\begin{equation} \label{eq:jiiHecke} sts = \l_{iij} tst + p_{iij} ts \end{equation}
\[{
\labellist
\small\hair 2pt
 \pinlabel {$= \l_{iij}$} [ ] at 28 20
 \pinlabel {$+$} [ ] at 68 20
 \pinlabel {\tiny $p_{iij}$} [ ] at 88 32
\endlabellist
\centering
\ig{1.3}{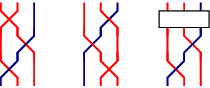}
}\]
for some $p_{iij} \in R_{iij}$ and some invertible scalar $\l_{iij}$.

If $\s_{ii}$ and $\s_{ij}$ and $\s_{ji}$ are permitted then one has the following relation in $\End(iji)$.
\begin{equation} \label{eq:ijiHecke} sts = \l_{iji} tst + q_{iji} \id_{iji} \end{equation}
\[{
\labellist
\small\hair 2pt
 \pinlabel {$= \l_{iji}$} [ ] at 28 20
 \pinlabel {$+$} [ ] at 68 20
 \pinlabel {\tiny $q_{iji}$} [ ] at 88 32
\endlabellist
\centering
\ig{1.3}{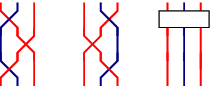}
}\]
for some $q_{iji} \in R_{iji}$ and some invertible scalar $\l_{iji}$.

If $\s_{ij}$, $\s_{ik}$, and $\s_{jk}$ are all permitted, then one has the following relation in $\Hom(ijk,kji)$.
\begin{equation} \label{eq:ijkHecke} sts = \l_{kji} tst. \end{equation}
\[{
\labellist
\small\hair 2pt
 \pinlabel {$= \l_{kji}$} [ ] at 28 20
\endlabellist
\centering
\ig{1.3}{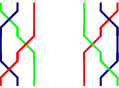}
}\]
for some invertible scalar $\l_{kji}$.

If $\s_{ii}$ is permitted, then one has the following relation in $\End(ii)$ for each $f \in R_{ii}$.
\begin{equation} \label{eq:sfHecke} \s_{ii} f = \phi_i(f) \s_{ii} + \pa_i(f) \id_{ii} \end{equation}
\[{
\labellist
\tiny\hair 2pt
 \pinlabel {$f$} [ ] at 7 9
 \pinlabel {$\phi_i(f)$} [ ] at 40 21
 \pinlabel {$\pa_i(f)$} [ ] at 72 21
 \pinlabel {\small $=$} [ ] at 23 14
 \pinlabel {\small $+$} [ ] at 55 14
\endlabellist
\centering
\ig{1.5}{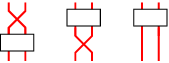}
}\]
for some linear maps $\phi_i \co R_{ii} \to R_{ii}$ and $\pa_i \co R_{ii} \to R_{ii}$.

If $\s_{ij}$ is permitted, then one has the following relation in $\Hom(ij,ji)$ for each $f \in R_{ij}$.
\begin{equation} \label{eq:ijfHecke} \s_{ij} f = \phi_{ij}(f) \s_{ij} \end{equation}
\[{
\labellist
\tiny\hair 2pt
 \pinlabel {$f$} [ ] at 7 9
 \pinlabel {\small $=$} [ ] at 23 14
 \pinlabel {$\phi_{ij}(f)$} [ ] at 40 21
\endlabellist
\centering
\ig{1.5}{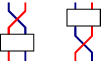}
}\]
for some linear map $\phi_{ij} \co R_{ij} \to R_{ji}$.

This concludes the list of relations.
\end{subequations} \end{defn*}

\begin{example} The Khovanov-Lauda-Rouquier algebras \cite{KhoLau09,Rouq2KM-pp} attached to a Cartan datum (and a matrix $Q$ of polynomials) are an example of a Hecke-type category. In this example, every crossing $\s_{ij}$ is permissible. The dot rings $R_i$ are each polynomial rings in one variable (with different gradings). \end{example}
	
\begin{example} The Webster algebras \cite{WebsKIHRT} (for simplicity, let us do the case of $\sl_2$) are an example of a Hecke-type category. The set $I$ of colors consists of one kind
of black strand, and one kind of red strand for each $\l \in \Z_{\ge 0}$. It is not permissible to cross two red strands, but other crossings are permissible. $R_{\textrm{black}}$ is a
polynomial ring in one variable, while $R_{\textrm{red}_\l}$ is just the ground field, for any $\l \ge 0$. \end{example}

\begin{thm} \label{thm:MDLforHecke} Let $\CC$ be a Hecke-type category. Then $\Hom(\un{i}, \un{j})$ has a basis (over dots $R_{\un{j}}$ acting on the left) indexed by permissible $w \in S_n$, obtained by taking an arbitrary
reduced expression for each $w$, if and only if: \begin{itemize} \item Each of the ambiguities in \eqref{eq:Samb} is resolvable, for each permissible coloring. \item Each of the
ambiguities below is resolvable, for each possible coloring. \begin{equation} \label{eq:Sambf} \ig{1.3}{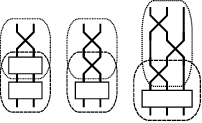} \end{equation} \end{itemize} \end{thm}
	
\begin{proof} One can easily modify the relation of Definition \ref{def:partial} to account for the extra data of the coloring. Given Lemma \ref{lem:Sambenough}, it is easy to argue
that the ambiguities in \eqref{eq:Samb} and \eqref{eq:Sambf} combine to give a minimal set of ambiguities. Then Theorem \ref{thm:MDL}, Theorem \ref{thm:sink}, and the discussion of
\S\ref{subsec:noncentral} combine to give this theorem. \end{proof}

As stated earlier, in this paper we will not compute what compatibilities between these coefficients are required in order that the minimal ambiguities of \eqref{eq:Samb} and
\eqref{eq:Sambf} are resolvable. Under some simplifying assumptions, this will be done in \cite{EGarcia1}. Note that, for Webster algebras and Khovanov-Lauda-Rouquier algebras, checking the ambiguities is greatly simplified by the fact that $a_i = b_i = c_i = d_i = e_i = 0$ in \eqref{eq:iiiHecke}, so that it becomes a relatively easy exercise (and worthwhile practice for the neophyte).

\begin{remark} Checking the ambiguity $\s_{ii} m_f m_g$ will imply that $\phi_i$ is a homomorphism of rings. Similarly, so is $\phi_{ij}$. The map $\pa_i$ is not a homomorphism, instead satisfying some kind of modified Leibniz rule.

If the maps $\phi_i$ and $\phi_{ij}$ are isomorphisms, then one can use these same relations to move dots to the right rather than the left. Then $\CC$ will have the same basis over
dots acting on the right. \end{remark}

\subsection{A further direction}
\label{subsec:question}

We end this paper by posing a question for the interested student. Manin-Schechtmann theory is a statement only about symmetric groups, which are the Coxeter groups of type $A$. It has
begun to be generalized to type $B$, see \cite{SheSuh}, though one does not expect that it will generalize to other types, even type $D$ (see \cite{EThick} for further discussion).
Hence, we are no closer to a Bergman diamond lemma which works for (modified) Coxeter group algebras in general.

The Bergman diamond lemma relies upon a partial order, and upon choosing a direction for each relation. However, it is reasonable to ask whether there might be a modification of the
diamond lemma which works for preorders, where certain relations can be applied in both directions. For example, in a Coxeter presentation, perhaps one wishes to apply the braid
relations in either direction, but requires the quadratic relation to be applied as $\s\s \mapsto 1$. Working with preorders creates new ``circular ambiguities,'' as one uses relations
to travel in circular paths around equivalent words in this preorder. There are other new kinds of ambiguities as well: two different words in the same equivalence class might admit a
relation, giving two possibly different ways of resolving the equivalence class.
 
To resolve the possible (infinite) chains of resolutions which can now occur, one needs a fairly strong understanding of the equivalence classes. For Coxeter groups, we know a lot about
reduced expression graphs: all their monodromy (i.e. generators of $\pi_1$ of the graph) comes from finite parabolic subgroups of rank 3, and appearances of their longest elements as
subwords. For example, we have already seen the loop in the reduced expression graph of $stsuts$: follow \eqref{eq:path1} forwards then \eqref{eq:path2} backwards. Perhaps there is a variant on the diamond lemma which applies to this preorder, and for which the minimal ambiguities are classifiable. 

\begin{remark} There is also monodromy in reduced expression graphs coming from disjoint ambiguities (i.e. applying braid relations to disjoint parts of a word), which are always
resolvable. \end{remark}

Having such a variant on the diamond lemma would be more philosophically satisfying, as choosing particular reduced expressions (like the Manin-Schechtmann sink) seems like an
unnatural thing to be forced to do.

{\bf Addendum}. Circular ambiguities should be treatable using the same ``non-terminating'' techniques as were showcased in \cite{Dupont}.


\bibliographystyle{plain}
\bibliography{mastercopy}

\end{document}